\documentclass[11pt,a4paper]{article}
\usepackage[T1]{fontenc}
\usepackage[utf8]{inputenc}
\usepackage{lmodern,microtype}
\usepackage[a4paper,margin=27mm]{geometry}
\usepackage{amsmath,amssymb,amsthm,mathtools}
\usepackage{booktabs,longtable,array,enumitem}
\usepackage{tikz}
\usetikzlibrary{arrows.meta,positioning,calc}
\usepackage{xcolor,xurl,fancyvrb}
\usepackage[unicode,colorlinks=true,linkcolor=blue!50!black,citecolor=blue!50!black,urlcolor=blue!50!black]{hyperref}
\hypersetup{pdftitle={Oracle-Tree Forcing and Wilf's Inequality for Thirteen Left Elements},pdfauthor={Michel Gaspar}}
\newtheorem{theorem}{Theorem}[section]
\newtheorem{lemma}[theorem]{Lemma}
\newtheorem{proposition}[theorem]{Proposition}
\newtheorem{corollary}[theorem]{Corollary}
\theoremstyle{definition}\newtheorem{definition}[theorem]{Definition}
\theoremstyle{remark}\newtheorem{remark}[theorem]{Remark}
\numberwithin{equation}{section}
\newcommand{\N}{\mathbb N}

\newcommand{\PP}{\mathbb P}
\newcommand{\QQ}{\mathbb Q}
\newcommand{\FF}{\mathcal F}
\newcommand{\II}{\mathcal I}
\newcommand{\MM}{\mathcal M}
\newcommand{\RO}{\operatorname{RO}}
\newcommand{\Succ}{\operatorname{Succ}}
\newcommand{\stem}{\operatorname{stem}}
\newcommand{\supp}{\operatorname{supp}}
\newcommand{\cov}{\operatorname{cov}}
\newcommand{\non}{\operatorname{non}}
\newcommand{\add}{\operatorname{add}}
\newcommand{\cof}{\operatorname{cof}}
\newcommand{\dens}{\operatorname{dens}}
\newcommand{\pw}{\pi\mathrm w}

\setlist[enumerate]{itemsep=3pt,topsep=4pt}

\fvset{fontsize=\small}
\title{Oracle-Tree Forcing and Wilf's Inequality\\
for Thirteen Left Elements}
\author{Michel Gaspar}
\date{25 September 2026}
\begin{document}
\maketitle
\begin{abstract}
We study filter-Laver forcing on $\omega^2$ with successor sets containing
final quadrants. The observation retaining the minimum coordinate and
orientation is a complete projection. Its finite nontrivial hidden-label
quotients are Cohen, whereas the full offset quotient has Boolean density
equal to the ground-model dominating number $\mathfrak d^V$. For the associated grid-meager
ideal $\mathcal I_\square$, we prove
$\operatorname{cov}(\mathcal I_\square)=\operatorname{add}(\mathcal M)$
for the usual meager ideal $\mathcal M$, and obtain further cardinal bounds.
We also analyze threshold games,
oracle degrees and the distinction between positive and filter-large
certificate sets: recurrent filter-large acceptance admits refinements
whose every branch succeeds.
As an arithmetic application, we prove Wilf's inequality $c\le13e$ for
numerical semigroups with thirteen elements below the conductor $c$,
where $e$ is the embedding dimension. A complete
arithmetic proof, including a bounded exact computation, is given in the
appendix. Its terminating certificate procedures supply a computable
quadrant tree and uniform refinements for every forcing condition. Finite
support kernels describe persistence of the registered proofs; arithmetic
absoluteness explains precisely why forcing organizes these certificates
without replacing their arithmetic justification.
\end{abstract}
\tableofcontents

\part{Quadrant oracle forcing}
\section{Introduction}

For a countably infinite alphabet $A$ and a free filter $\FF$ on $A$, filter-Laver
forcing $L_{\FF}(A)$ consists of stemmed trees whose successor set above
the stem belongs to $\FF$ \cite{khomskii}. Our example takes $A=\omega^2$
and requires every such successor set to contain a final quadrant. The
forcing itself is therefore an instance of an established construction.
What matters here is its specified observation, which retains the minimum
coordinate and orientation but forgets the offset. The residual order has
a sharp change when the set of hidden labels passes from finite to
countably infinite.

We introduce the finite support language needed here under the name
\emph{Machine Calculus}. A resource support records which registered
premises and derivation artifacts are needed for an instruction.
The definitions and their support equation are given below, so no
external account of this calculus is required. At a forcing node, withdrawing finitely
many resource tokens determines a filter of surviving instructions.
This incidence construction can be studied independently of any claim
that a particular program has supplied sound derivation artifacts.
The distinction is useful: a set-theoretic projection is a theorem
about orders, whereas a Turing readout also requires uniform algorithms.
Part~I develops the projection, quotient and ideal analysis, then
compares the order through games, finite statistical laws, and effective
certificates. Its quotient-density and grid-ideal theorems form a
set-theoretic analysis independent of the arithmetic application.
Part~II specializes these tools to a finite arithmetic
application with an independent proof in the appendix. Readers seeking
only the numerical-semigroup result can begin at
Section~\ref{sec:wilf-intro} and then read the arithmetic appendix
from Section~\ref{sec:foundations}, without the forcing sections.
We use the threshold-game construction of Miller and the filter-tree
framework of Khomskii \cite{miller,khomskii}, and give the specific
lifting and strategy arguments needed for the grid order.

We work in ZFC and write $V$ for a transitive ground universe in
forcing arguments. This is the usual model language for the forcing
theorem, not an additional existence assertion about transitive models.
A superscript $V$ on a cardinal denotes its value in the ground model.
Trees and their codes belong to $V$ unless stated otherwise. We use
``complete projection'' in the standard lifting sense, without
surjectivity onto individual conditions. We do not identify this order
with the stem-and-bound presentation of Hechler forcing; the tree
presentation and that presentation can differ \cite{palumbo}.

\subsection*{Notation}
We write $f\leq^*g$ for eventual domination in $\omega^\omega$.
The cardinals $\mathfrak b$ and $\mathfrak d$ are the least sizes of
an unbounded and a dominating family, respectively, for this order;
$\mathfrak c=2^{\aleph_0}$. The pseudointersection number $\mathfrak p$
is the least size of a family of infinite subsets of $\omega$ with
the strong finite-intersection property (each finite intersection is
infinite) and no infinite set almost
contained in every member. Let $\MM$ be the usual meager ideal on
$2^\omega$. For a proper ideal $I$ on a set $Z$, $\add(I)$ is the
least size of a subfamily of $I$ whose union is outside $I$;
$\cov(I)$ is the least size of one covering $Z$; $\non(I)$ is the
least size of an $I$-positive subset of $Z$; and $\cof(I)$ is the
least size of a cofinal subfamily of $(I,\subseteq)$.
We use the standard equality of $\cov(\MM)$ computed on Cantor
space and on Baire space.

For a partial order $P$, $\RO(P)$ is its regular-open Boolean
completion, $\mathbf2$ is the two-element Boolean algebra, and
$\dens(P)$ is the least size of a dense subset of its nonzero part.
For a topological space $(Z,\tau)$, $\dens(Z,\tau)$ is its least
dense-set size and $\pw(Z,\tau)$ is its $\pi$-weight, the least
size of a family of nonempty open sets refining every nonempty
open set. A family of conditions is \emph{centered} if every
finite subfamily has a common extension; a countable union of
centered families is $\sigma$-centered. The countable chain
condition (ccc) means every antichain is countable. For a finite
alphabet $K$, $K^{<\omega}$ is ordered by end extension.
Cohen forcing is $2^{<\omega}$ with this order; its Boolean
completion is the Cohen algebra.
For real codes $x,y$, $x\leq_Ty$ means computability with oracle
$y$, $x\equiv_Ty$ means reducibility both ways, and $x\oplus y$
is their computable join.
For any set $Z$, $[Z]^{<\omega}$ denotes its finite subsets.

\section{Resources, observations, and tree orders}

\begin{definition}[Finite support kernel]
Fix a formula language and a proof environment $\kappa$;
$\Gamma\vdash_\kappa\varphi$ means that $\varphi$ has a checked
derivation from assumptions $\Gamma$ in that environment.
A finite support kernel $K$ has disjoint finite sets $\mathcal P$ of
premise identifiers, $\mathcal V$ of claim identifiers, and $\mathcal E$
of artifact identifiers. Each $p\in\mathcal P$ is assigned a formula
$\psi_p$, and each $v\in\mathcal V$ is assigned a formula $\varphi_v$.
Each $e\in\mathcal E$ carries finite sets
$A_e\subseteq\mathcal P$ and $B_e\subseteq\mathcal V$, a conclusion
$c_e\in\mathcal V$, and a checked conditional derivation
\[
 \{\psi_p:p\in A_e\}\cup\{\varphi_v:v\in B_e\}
       \vdash_\kappa\varphi_{c_e}.
\]
The \emph{resource-token space} is the tagged disjoint union
$\mathcal U_K=\mathcal P\sqcup\mathcal E$. A complete finite
derivation tree $t$ expands every prerequisite in $B_e$; its support
$\supp(t)\subseteq\mathcal U_K$ is the union of the artifacts and premise
identifiers used in the tree. For an available-resource set
$W\subseteq\mathcal U_K$,
$\operatorname{Act}_K(W)$ is the least closure of claims under the
artifacts whose guards and artifact tokens lie in $W$; explicitly,
with $V_0=\varnothing$ and
\[
 V_{j+1}=V_j\cup
 \{c_e:e\in\mathcal E\cap W,\ A_e\subseteq W,\
                          B_e\subseteq V_j\},
 \qquad
 \operatorname{Act}_K(W)=\bigcup_{j<\omega}V_j.
\]
\end{definition}

The finite closure gives the familiar support equation
\begin{equation}\label{eq:support-equation}
 v\in\operatorname{Act}_K(W)
 \quad\Longleftrightarrow\quad
 \exists t\,[t\text{ is a complete derivation of }v
             \ \&\ \supp(t)\subseteq W].
\end{equation}
Indeed, induction on the first closure stage constructs $t$, and
induction on the height of $t$ gives the converse. The soundness of a
conclusion still depends on the truth of the premise formulas in $t$.
An observation on a set $\mathcal H$ of instruction histories is
a map $q:\mathcal H\to Y$. If a context map
$c:\mathcal H\to C$ is present, $q$ factors through $c$ when
there is $\bar q:c[\mathcal H]\to Y$ with $q=\bar q\circ c$.
Equivalently, $c(h)=c(h')$ implies $q(h)=q(h')$. This
set-theoretic quotient is
the motivation for the symbol observation below; complete projection
requires an additional lifting proof.

Hereafter let $A$ be a countably infinite instruction alphabet and
$U=\{\tau_j:j<\omega\}$ a countably infinite resource-token space,
with no repetitions in the enumeration. Let
$s:A\to[U]^{<\omega}\setminus\{\varnothing\}$ be a support assignment.
For finite $F\subseteq U$, set
\[
 X_F=\{a\in A:s(a)\cap F=\varnothing\},\qquad
 \FF_s=\{X\subseteq A:\exists F\in[U]^{<\omega}\ X_F\subseteq X\}.
\]
Assume every $X_F$ is nonempty. Since
$X_{F\cup G}=X_F\cap X_G$, this generates a proper filter.
It is free: it contains every cofinite set and no finite set.
Indeed, one may exclude a prescribed finite set of instructions by
withdrawing one token from each of their nonempty supports; properness
then rules out finite members of the filter.
For a free filter $\FF$ on a countably infinite alphabet $A$, the
conditions of $L_{\FF}(A)$ are nonempty, downward-closed trees
$T\subseteq A^{<\omega}$ with a stem $\stem(T)$, the longest
node common to all branches: below it the tree
has one path and, for every $t\supseteq\stem(T)$ in $T$,
$\Succ_T(t)=\{a:t^\frown a\in T\}\in\FF$. Here $A^{<\omega}$
is the set of finite strings, $t^\frown a$ appends $a$, and
$[T]=\{x\in A^\omega:\forall n\ (x\restriction n\in T)\}$
is the body of $T$. For $u\in T$, the cone
$T\restriction u=\{t\in T:t\subseteq u\text{ or }u\subseteq t\}$.
The forcing order is $S\leq T$ exactly when $S\subseteq T$;
smaller trees are stronger. The full tree is its largest
condition $1_{L_{\FF}(A)}$.

\begin{lemma}[Rank-tail normal form]\label{lem:rank}
Put $r(a)=\min\{j:\tau_j\in s(a)\}$. The family $\FF_s$ is a proper filter
if and only if $r[A]$ is unbounded in $\omega$. In that case
\[
 \FF_s=\{X\subseteq A:\exists N\ \{a:r(a)\ge N\}\subseteq X\}.
\]
If every rank fiber is finite, the dual ideal
$\{A\setminus X:X\in\FF_s\}$ is
$\mathrm{Fin}$, the ideal of finite subsets of $A$.
\end{lemma}
\begin{proof}
For $F_N=\{\tau_0,\ldots,\tau_{N-1}\}$, $X_{F_N}=\{a:r(a)\ge N\}$.
Every finite $F$ lies in some $F_N$, so these are cofinal among
the generators. They are all nonempty precisely when $r[A]$ is
unbounded. Complementation gives the last assertion.
\end{proof}

Let $B$ be nonempty and finite, $Y=\omega\times B$, and
$E_N^B=[N,\infty)\times B$. Write $\QQ_B$ for the stemmed-tree order
on $Y^{<\omega}$ whose successor sets above the stem contain some
$E_N^B$. For a symbol map $\rho:A\to Y$, let
$\pi_\rho(T)=\{\rho\circ t:t\in T\}$. We require
\begin{align}
\tag{S}\label{eq:S}
&\forall F\in[U]^{<\omega}\ \exists N\ \forall m\ge N\ \forall b\in B\
   \exists a\in X_F\ (\rho(a)=(m,b)),\\
\tag{P}\label{eq:P}
&\forall N\ \exists F_N\in[U]^{<\omega}\
       X_{F_N}\subseteq\rho^{-1}(E_N^B).
\end{align}
For any $q_0\in\QQ_B$, write
$\QQ_B\downarrow q_0=\{q\in\QQ_B:q\leq q_0\}$.

\begin{theorem}[Observation projection]\label{thm:projection}
If \eqref{eq:S} and \eqref{eq:P} hold, then
$q_0=(\rho[A])^{<\omega}\in\QQ_B$, and
$\pi_\rho:L_{\FF_s}(A)\to\QQ_B\downarrow q_0$ is order preserving,
sends the largest condition to $q_0$, and satisfies
\[
 \forall T\ \forall q\le\pi_\rho(T)\ \exists S\le T\
       \bigl(\pi_\rho(S)\le q\bigr).
\]
Thus it is a complete projection onto the indicated cone. If $\rho$
is onto $Y$, the target is all of $\QQ_B$.
\end{theorem}
\begin{proof}
By (S), $\rho[A]$ contains a full tail. If $X_F\subseteq\Succ_T(t)$,
then $\rho[\Succ_T(t)]$ contains a target tail, so $\pi_\rho(T)$ is a
condition with stem $\rho\circ\stem(T)$: the target tail there
prevents a longer common stem. Given $q\le\pi_\rho(T)$, lift $\stem(q)$ to some
$t\in T$, cone $T$ at $t$, and retain precisely extensions $u$
whose projected histories lie in $q$. At a retained node $u$, choose
$F$ with $X_F\subseteq\Succ_T(u)$ and $N$ with
$E_N^B\subseteq\Succ_q(\rho\circ u)$. By (P), choose $G$ with
$X_G\subseteq\rho^{-1}(E_N^B)$. Then
$X_{F\cup G}\subseteq\Succ_S(u)$, so the retained tree $S$ is a
condition and $\pi_\rho(S)\le q$.
\end{proof}

\begin{remark}
The target cone cannot in general be replaced by all of $\QQ_B$.
The map $(m,b)\mapsto(m+1,b)$ on a cofinite-tail alphabet satisfies
(S) and (P), but no image stem can begin with $(0,b)$.
\end{remark}

For a complete projection $\pi:P\to Q$ and a $Q$-generic filter
$H$ over $V$, the quotient order used below is
\[
 P/H=\{p\in P^V:\pi(p)\in H\},
\]
ordered as in $P$. We write $P\simeq Q$ when the two orders have
isomorphic dense suborders.

\section{The quadrant order and the hidden-label phase}\label{sec:quadrant}

Set $A_\square=\omega^2$, $E_N=[N,\infty)^2$, and
\[
 \FF_\square=\{X\subseteq\omega^2:\exists N\ E_N\subseteq X\},
 \qquad \PP_\square=L_{\FF_\square}(\omega^2).
\]
To realize this filter by tokens, take the countable resource-token
space
\[
 U_\square=\{r_m:m<\omega\}\sqcup
            \{e_{i,j}:(i,j)\in\omega^2\},
 \qquad
 s_\square(i,j)=\{r_{\min(i,j)},e_{i,j}\}.
\]
Finite withdrawals leave a final quadrant, and withdrawing
$r_0,\ldots,r_{N-1}$ leaves exactly $E_N$; hence
$\FF_{s_\square}=\FF_\square$.
At forcing level $n$ we use a disjoint copy
\begin{align*}
 U_{n,\square}
   &=\{r_{n,m}:m<\omega\}\sqcup
     \{e_{n,i,j}:(i,j)\in\omega^2\},\\
 s_{n,\square}(i,j)
   &=\{r_{n,\min(i,j)},e_{n,i,j}\},\\
 U_{\infty,\square}
   &=\bigsqcup_{n<\omega}U_{n,\square}.
\end{align*}
Each level has the same induced filter $\FF_\square$.
The filter is $F_\sigma$ in $2^{\omega^2}$:
\[
 \FF_\square=\bigcup_{N<\omega}
      \{X\subseteq\omega^2:E_N\subseteq X\}.
\]
Each set in this union is closed in the product topology. The visible alphabet is
$Y=\omega\times2$ and
\[
 \rho_\square(i,j)=\bigl(\min\{i,j\},\mathbf1_{i>j}\bigr),
 \qquad \QQ=\QQ_2.
\]
Write $\pi=\pi_{\rho_\square}$ for this coordinatewise
observation on trees.
The preimage of $E_N^2$ contains $E_N$; the image of $E_N$ is exactly
$E_N^2$. Consequently Theorem~\ref{thm:projection} applies to
$s_\square$ and $\rho_\square$.
The order $\PP_\square$ is $\sigma$-centered: trees with a common
stem have a common strengthening, obtained by finite intersection.
The same is true of $\QQ$. A $\PP_\square$-generic branch
$g\in(\omega^2)^\omega$ has visible branch $h=(m,b)=\rho_\square\circ g$.
The minimum $m$ dominates every ground-model function by pruning
successor quadrants. The orientation $b$ is Cohen over the visible
minimum extension, as follows from the finite product presentation
in the next proof.

For a nonempty finite or countably infinite label set $K$
(unrelated to the finite kernel above), put
$A_K=\omega\times2\times K$,
$E_N^K=[N,\infty)\times2\times K$, and
$\PP_K=L_{\FF_K}(A_K)$ where
$\FF_K=\{X\subseteq A_K:\exists N\ E_N^K\subseteq X\}$.
Let $\rho_K(m,b,k)=(m,b)$. At machine stage $n$ one may use distinct
resource tokens $r_{n,m}$ and $e_{n,m,b,k}$. More precisely, the
stage-$n$ token space and the full token space are
\[
 \mathcal U_{n,K}
 =\{r_{n,m}:m<\omega\}\sqcup
  \{e_{n,m,b,k}:m<\omega,\ b<2,\ k\in K\},
 \qquad
 \mathcal U_{\infty,K}
 =\bigsqcup_{n<\omega}\mathcal U_{n,K}.
\]
The support map at stage $n$ is
\begin{equation}\label{eq:stage-support}
 s_n:A_K\longrightarrow[\mathcal U_{n,K}]^{<\omega},
 \qquad s_n(m,b,k)=\{r_{n,m},e_{n,m,b,k}\}.
\end{equation}
At any fixed stage, finite withdrawal from that stage's token
space generates exactly $\FF_K$:
removing finitely many tokens leaves a full block $E_N^K$, while
withdrawing all row tokens with $m<N$ leaves exactly $E_N^K$.
The forcing tree uses the stage-$n$ copy of this filter at level $n$;
its local withdrawal witness may vary from node to node. Thus the
token space records the resource incidence, whereas $A_K$ is the
alphabet of successor instructions.
This is an incidence assertion. It becomes a Machine Calculus
implementation only when each $e_{n,m,b,k}$ is equipped with a
checked conditional derivation in a specified proof environment,
its premise guards can be tested, and the stage palettes and
derivation codes are uniformly computable.

\begin{theorem}[Hidden-label phase]\label{thm:phase}
For every nonempty finite or countably infinite $K$, the map
$\pi_K:\PP_K\to\QQ$ induced by $\rho_K$ is a complete projection.
If $H\subseteq\QQ$ is generic over $V$, then in $V[H]$
\[
 \RO(\PP_K/H)\cong
 \begin{cases}
  \mathbf2,&|K|=1,\\
  \RO(K^{<\omega}),&2\le|K|<\omega,\\
  \RO(\PP_\square/H),&|K|=\aleph_0.
 \end{cases}
\]
In particular the finite nontrivial quotient is Cohen; the last
quotient has Boolean density $\mathfrak d^V$ and is not Cohen.
\end{theorem}
The maps used here fit into the commuting observation diagram
\[
\begin{tikzpicture}[baseline=(current bounding box.center),>=Stealth,
 every node/.style={inner sep=2pt}]
\node (grid) at (0,1.45) {$\PP_\square$};
\node (finite) at (4,1.45) {$\PP_{K_k}$};
\node (visible) at (2,0) {$\QQ$};
\draw[->] (grid) -- node[above] {$\alpha_k$} (finite);
\draw[->] (grid) -- node[below left] {$\pi_\square$} (visible);
\draw[->] (finite) -- node[below right] {$\pi_{K_k}$} (visible);
\end{tikzpicture}
\]
Here $K_k=\{0,\ldots,2^k-1\}$ and $\alpha_k$ records the offset
modulo $2^k$, as defined in Corollary~\ref{cor:noncohen}.
\begin{proof}
The projection hypotheses follow from the two exact block relations
$\rho_K[E_N^K]=E_N^2$ and
$E_N^K\subseteq\rho_K^{-1}(E_N^2)$.
For finite $K$, let $C=2\times K$ and let $D$ be the stemmed-tree
order on $\omega^{<\omega}$ whose successor sets above a stem
contain a final interval. Synchronizing stem lengths gives
a dense suborder of $D\times C^{<\omega}$. A pair $(D_0,\eta)$ of equal
stem length maps to the paired tree with minimum projection $D_0$,
fixed initial label string $\eta$, and every $C$-label continuation.
Its range is dense in $\PP_K$: at each minimum node $v$, only
finitely many label histories are present, so the maximum of their
finitely many successor thresholds supplies a common minimum tail.
Thus
\[
 \PP_K\simeq D\times2^{<\omega}\times K^{<\omega},
 \qquad \QQ\simeq D\times2^{<\omega},
\]
and $\pi_K$ drops the third factor on these dense orders. In the
$\QQ$-generic extension the residual completion is
$\RO(K^{<\omega})$. For $|K|=1$ it is trivial; for finite
$|K|\ge2$ it is the completion of a countable atomless forcing,
hence the Cohen algebra.

For countably infinite $K$, fix a bijection $\nu:K\to\omega$ and
define
\[
 \theta(m,0,k)=(m,m+\nu(k)),\qquad
 \theta(m,1,k)=(m+\nu(k)+1,m).
\]
This bijection sends $E_N^K$ exactly to $E_N$, and
$\rho_\square\circ\theta=\rho_K$. Coordinatewise application is
therefore an order isomorphism commuting with the projections. The
density assertion follows from Theorem~\ref{thm:density} below.
\end{proof}

\section{The full offset quotient}

Let $H\subseteq\QQ$ be generic and $h=(m,b)$ its branch. A target
generic filter is recovered from its branch: if a ground target tree
contains $h$, a sufficiently long cone of that tree is compatible
with any target condition containing $h$, and the dense compatibility
decision puts the tree in $H$. Thus $V[H]=V[m,b]$. In this model let
\[
 R_H=\PP_\square/H
    =\{T\in\PP_\square^V:\pi(T)\in H\}.
\]
For $d\in\omega^\omega$ with $b(n)=1\Rightarrow d(n)>0$, reconstruct
the full branch by
\begin{equation}\label{eq:reconstruct}
 g_{h,d}(n)=
 \begin{cases}
  (m(n)+d(n),m(n)),&b(n)=1,\\
  (m(n),m(n)+d(n)),&b(n)=0.
 \end{cases}
\end{equation}
Put
$X_h=\{d\in\omega^\omega:\forall n\,
 (b(n)=1\Rightarrow d(n)>0)\}$, the closed product of
admissible offsets, and
$K_T(h)=\{d\in X_h:g_{h,d}\in[T]\}$.

\begin{proposition}[Closed-fiber order]\label{prop:fiber}
For each $T\in R_H$, $K_T(h)$ is nonempty and closed in the ordinary
product topology. For $T,U\in R_H$,
\[
 T\parallel U\quad\Longleftrightarrow\quad
 K_T(h)\cap K_U(h)\ne\varnothing,
 \qquad
 T\leq_{\mathrm{sep}}U\quad\Longleftrightarrow\quad
 K_T(h)\subseteq K_U(h),
\]
where $T\parallel U$ means compatibility and
$T\leq_{\mathrm{sep}}U$ means that every $S\leq T$ is compatible
with $U$. Hence the separative
quotient of $R_H$ is the inclusion order of its nonempty
\emph{ground-coded} closed fibers.
\end{proposition}
\begin{proof}
Closedness follows because membership in $[T]$ is decided by all
finite prefixes. A quotient-generic extension below $T$ has a
branch through $T$ over $h$. If the countable-alphabet tree for
$K_T(h)$ had no branch already in $V[H]$, it would be well founded
there and carry an ordinal rank; further forcing cannot create a
branch through a well-founded ground tree. Thus $K_T(h)\ne\varnothing$.
If two fibers share $d$, their reconstructed branch belongs to both
trees. Cone their intersection beyond both stems; the intersection
still contains a quadrant at each node and its projection lies in
$H$. The converse follows by taking a point in the fiber of a
common extension. If $K_T(h)\subseteq K_U(h)$, every strengthening
of $T$ has fiber intersecting $K_U(h)$. If inclusion fails, choose
$d\in K_T(h)\setminus K_U(h)$. Closedness of $K_U(h)$ gives a finite
grid prefix $u$ of $g_{h,d}$ outside $U$. The ground cone
$T\restriction u$ survives the quotient and is incompatible with
$U$, so $T\not\leq_{\mathrm{sep}}U$.
\end{proof}

\begin{theorem}[Exact residual density]\label{thm:density}
\[
 \QQ\Vdash\dens\bigl(\RO(R_H)\bigr)=\check{\mathfrak d}^{\,V}.
\]
Here density means the least size of a dense subset of the nonzero
Boolean completion.
\end{theorem}
The upper bound below uses a ground-model dominating family to
produce dense pure-quadrant trees. The lower bound constructs one
tree $S_f$ that defeats every proposed smaller family: a shared
visible successor first chooses an offset $k$, then its required next
minimum exceeds the corresponding old threshold.
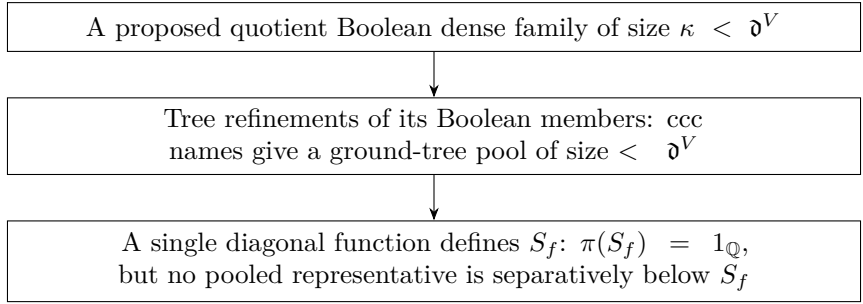
\begin{center}
\begin{tikzpicture}[>=Stealth,node distance=6mm,
 box/.style={draw,align=center,text width=110mm,inner sep=4pt,font=\small}]
\node[box] (names) {A proposed quotient Boolean dense family of size
 $\kappa<\mathfrak d^V$};
\node[box,below=of names] (pool) {Tree refinements of its Boolean members:
 ccc names give a ground-tree pool of size $<\mathfrak d^V$};
\node[box,below=of pool] (diag) {A single diagonal function defines $S_f$:
 $\pi(S_f)=1_{\QQ}$, but no pooled representative is separatively below $S_f$};
\draw[->] (names)--(pool);
\draw[->] (pool)--(diag);
\end{tikzpicture}
\end{center}
\begin{proof}
Enumerate $(\omega^2)^{<\omega}$ in $V$. An eventually dominating
family of size $\mathfrak d^V$, closed under finite modifications,
is pointwise cofinal among threshold maps
$F:(\omega^2)^{<\omega}\to\omega$. For every finite stem $s$ and
such $F$, let $T_{s,F}$ use exactly $E_{F(t)}$ above each node
$t\supseteq s$. These $\mathfrak d^V$ pure-quadrant trees are dense
in $\PP_\square$. They remain dense in $R_H$: below $T\in R_H$,
for each $q\le\pi(T)$ first lift $q$ below $T$, then choose a pure
tree below the lift. Genericity of $H$ meets the projections of
these choices. This proves the upper bound.

For the lower bound, fix $\{T_j:j\in J\}\subseteq\PP_\square^V$
with $|J|<\mathfrak d^V$ and choose a quadrant threshold $N_j(t)$
at each node of each tree. For every admissible $(j,t,M)$ with
$M\ge N_j(t)$ define, for $k\ge2$,
\[
 F_{j,t,M}(k)=N_j\bigl(t^\frown(M+k,M)\bigr).
\]
There are fewer than $\mathfrak d^V$ such functions. Choose $f_0$
not eventually dominated by any of them, and put
$f(0)=f(1)=0$, $f(k)=\max\{f_0(k),k\}$ for $k\ge2$.
Let $S_f$ have full root successors and require at every later
stage $m(n+1)\ge f(d(n))$, where $d(n)$ is the offset of the
preceding grid symbol. More precisely, a finite grid string
$v$ lies in $S_f$ exactly when
\[
 \forall n\ (n+1<|v|\ \Longrightarrow\
 \min(v(n+1))\ge
 f(|v(n)_0-v(n)_1|)).
\]
This is a grid condition and
$\pi(S_f)=1_{\QQ}$: every finite visible history lifts using
the canonical offsets $d(n)=b(n)\in\{0,1\}$.

Fix $j$ and any $q\le\pi(T_j)$. For nodes above its stem choose
thresholds $N_q(v)$ with $E_{N_q(v)}^2\subseteq\Succ_q(v)$.
Lift a node $s\in q$ beyond the
stems to $t\in T_j$, and fix
$M\ge\max\{N_j(t),N_q(s)\}$. For $k\ge2$ let
$t_k=t^\frown(M+k,M)\in T_j$; all $t_k$ project to the same
$s'=s^\frown(M,1)\in q$. Put $B=N_q(s')$. Choose $k>B$ with
$f(k)>F_{j,t,M}(k)$ and set
$L=\max\{F_{j,t,M}(k),B\}<f(k)$.
Then $u=t_k^\frown(L,L)\in T_j$ and $\rho_\square\circ u\in q$, while
$u\notin S_f$. Coning $T_j$ at $u$ and retaining projected
histories in a cone of $q$ yields a ground condition $U\le T_j$
with $\pi(U)\le q$ and $U\perp S_f$. Thus visible conditions forcing
that $T_j$ is not separatively below $S_f$ are dense below
$\pi(T_j)$.

If a condition forced a quotient Boolean dense family of size
$\kappa<\mathfrak d^V$, choose for each nonzero Boolean member a
name for a ground-tree quotient condition below it. The ccc of
$\QQ$ gives countably many possible ground values for each such
name, decided by a maximal antichain.
Their union has ground size at most
$\max\{\kappa,\aleph_0\}<\mathfrak d^V$. The preceding diagonal
construction makes none of these representatives separatively
below $S_f$ in the visible generic extension, since each failure
is forced densely below its projection. As $\pi(S_f)=1_\QQ$,
$S_f$ is a nonzero quotient condition, contradicting Boolean
density. Thus the lower bound holds. Since
$\QQ$ is ccc, $\mathfrak d^V$ remains a cardinal in $V[H]$.
\end{proof}

\begin{corollary}\label{cor:noncohen}
The countably infinite hidden-label quotient is not Cohen-equivalent.
It is not $\omega^\omega$-bounding: finite residues of the offset
give Cohen reals over $V[H]$.
\end{corollary}
\begin{proof}
The Cohen algebra has a countable dense set, whereas
$\mathfrak d^V>\aleph_0$. For fixed $k\ge1$, put
$K_k=\{0,\ldots,2^k-1\}$ and define
\[
 \alpha_k(i,j)=
 \bigl(\min\{i,j\},\mathbf1_{i>j},|i-j|\bmod 2^k\bigr).
\]
For every $N$,
$\alpha_k[E_N]=E_N^{K_k}$ and
$\alpha_k^{-1}(E_N^{K_k})=E_N$.
The proof of Theorem~\ref{thm:projection}, applied to these
block relations, makes $\alpha_k:\PP_\square\to\PP_{K_k}$
a complete projection commuting with the projections to
$\QQ$. Thus the residue sequence $d(n)\bmod 2^k$ is the
generic label sequence for the intermediate
$\PP_{K_k}/H$ factor. Theorem~\ref{thm:phase} makes this factor
Cohen. A Cohen real yields an unbounded real by enumerating
the positions of one of its symbols.
\end{proof}

\section{The grid ideal and local covering}

This section gives the grid order's ideal and cardinal structure.
The certificate application in Part~II uses the order, projection,
and game results directly.

Let $X=(\omega^2)^\omega$. Here $\operatorname{Bor}(X)$ means
the Borel sets for the ordinary product topology.
For an ideal $I$ on $X$, $\operatorname{Bor}(X)/I$ denotes the
Boolean algebra of Borel sets modulo symmetric difference in $I$.
The bodies $[T]$, $T\in\PP_\square$,
form a basis for a topology $\tau_\square$. If $g\in[T]\cap[S]$,
cone their intersection at a sufficiently long prefix of $g$.
The successor sets still contain quadrants, so this cone is a
condition. Every ordinary product cylinder is also a basic open
set, and $\tau_\square$ refines the product topology.
Each body $[T]$ is ordinary closed. Put
\[
 \mathcal N_\square=
 \{A\subseteq X:\forall T\in\PP_\square\ \exists S\le T\
                  ([S]\cap A=\varnothing)\},
 \qquad
 \II_\square=\sigma(\mathcal N_\square).
\]
Thus $\mathcal N_\square$ is the family of nowhere-dense sets
for $\tau_\square$ and $\II_\square$ is its meager ideal. The
distinction matters: a countable union of grid-null sets need not
be grid-null.

The presentation below is the grid instance of the filter-Laver
ideal construction \cite[Definition~2.1, Lemma~2.2 and
Theorem~2.8]{khomskii}. We give its direct proof to make the
relationship between the two topologies explicit.

\begin{proposition}[Idealized presentation]\label{prop:idealized-presentation}
No nonempty $\tau_\square$-open set is in $\II_\square$.
The map
\[
 T\longmapsto[T]\bmod\II_\square
\]
preserves order and incompatibility and has dense image in the Boolean
algebra $\operatorname{Bor}(X)/\II_\square$ of ordinary Borel
sets modulo symmetric difference in the ideal. On separative quotients it
is a dense embedding.
\end{proposition}
\begin{proof}
To prove the first assertion, suppose $A_n$ is nowhere dense for each $n$.
Starting with $T_0=T$, recursively choose
$T_{n+1}\le T_n$ whose body misses $A_n$, and extend its stem to length
at least $n+1$. The nested stems determine a branch belonging to every
$[T_n]$: any finite prefix of the union belongs to all earlier trees.
This branch lies in $[T]$ and outside $\bigcup_n A_n$. Notice that the
intersection of the trees need not itself be a condition; a branch is all
that this argument requires.

Every ordinary Borel set has the Baire property for $\tau_\square$.
Indeed ordinary open sets are $\tau_\square$-open, and in any topology the
sets differing from an open set by a meager set form a $\sigma$-algebra:
countable unions preserve the property, and the complement of an open set
differs from its interior by a closed nowhere-dense boundary. If a Borel
set $B$ is not in $\II_\square$, write $B\mathbin\triangle O$
meager with $O$ open. Then $O$ is nonempty; select $[T]\subseteq O$.
It follows that $[T]\setminus B$ is meager and $[T]$ is not, so the image
of $T$ is nonzero and below the class of $B$. This proves density.

Compatible trees have an intersection containing the body of a condition,
which is nonmeager by the first assertion. Conversely a branch common to two
trees yields a condition by coning their intersection beyond both stems.
Thus incompatible trees have disjoint bodies. Finally, if
$T\not\le_{\rm sep}S$, some $U\le T$ is incompatible with $S$;
its nonzero image witnesses that the image of $T$ is not below that of $S$.
Conversely, if the image of $T$ is not below that of $S$, density supplies
$U$ whose nonzero image lies below $[T]\setminus[S]$ modulo the ideal.
It is compatible with $T$ and incompatible with $S$, by the preceding
compatibility equivalence. A common strengthening of $U$ and $T$ therefore
witnesses $T\not\le_{\rm sep}S$.
\end{proof}

\begin{theorem}[Ideal characteristics]
\label{thm:ideal}
\begin{align*}
 \add(\II_\square)&=\omega_1,&
 \cof(\II_\square)&=\mathfrak c,\\
 \mathfrak p&\le\cov(\II_\square)
             \le\add(\MM),&
 \cof(\MM)&\le\non(\II_\square)\le\mathfrak c,\\
 \dens(X,\tau_\square)&=\pw(X,\tau_\square)
              =\dens(\PP_\square)=\mathfrak d.
\end{align*}
The first-line equalities follow from Khomskii's filter-Laver
results \cite[Remark~5.12]{khomskii}, since $\FF_\square$ is
$F_\sigma$. Proposition~\ref{prop:idealized-presentation} supplies
the associated dense embedding independently.
\end{theorem}
\begin{proof}
The filter $\FF_\square$ is $F_\sigma$, so the cited filter-Laver
results give the additivity and cofinality equalities.
We prove the remaining bounds.
For $g\in X$ let $m_g(n)=\min(g(n)_0,g(n)_1)$ and, for
$f\in\omega^\omega$, put
$Z_f=\{g:f\not\le^*m_g\}$. Each $Z_f$ is grid-null:
below a condition, raise its successor quadrant at level $n$
above $f(n)$. An unbounded family of $f$'s covers $X$ by the
corresponding $Z_f$, so $\cov(\II_\square)\le\mathfrak b$.
If $|Y|<\mathfrak d$, choose $f$ not eventually dominated by
any $m_g$, $g\in Y$; then $Y\subseteq Z_f$, proving
$\non(\II_\square)\ge\mathfrak d$.

The orientation map $\beta:X\to2^\omega$ sends $g$ to
$\mathbf1_{g(n)_0>g(n)_1}$ coordinatewise. The preimage of an
ordinary nowhere-dense subset of $2^\omega$ is grid-null:
extend the stem to an orientation cylinder missing that set.
Consequently $\cov(\II_\square)\le\cov(\MM)$ and
$\non(\II_\square)\ge\non(\MM)$. Standard Cicho\'n
identities give
$\min\{\mathfrak b,\cov(\MM)\}=\add(\MM)$ and
$\max\{\mathfrak d,\non(\MM)\}=\cof(\MM)$.

The order $\PP_\square$ is $\sigma$-centered. Bell's theorem
\cite{bell} supplies a filter meeting fewer than $\mathfrak p$
dense sets. For fewer than $\mathfrak p$ grid-meager sets,
decompose each into countably many grid-null sets. Meet the
corresponding dense avoidance sets and all dense sets requiring
stem length at least $n$. The directed stems give a branch
outside every set. Hence $\mathfrak p\le\cov(\II_\square)$.

The pure-quadrant trees used in Theorem~\ref{thm:density} form a
$\pi$-base of size $\mathfrak d$. Selecting one point from each
gives a dense set of that size. Conversely, if $Y\subseteq X$
has size below $\mathfrak d$, choose $f$ not eventually dominated
by any $m_g$, $g\in Y$. The condition requiring
$m_g(n)\ge f(n)$ at every stage misses $Y$. Thus both topological
cardinals are at least $\mathfrak d$, and the equalities follow.
\end{proof}

The upper bound on $\cov(\II_\square)$ is exact. We include the
reduction to Hru\v{s}\'ak--Minami's Martin-number theorem
\cite{hrusakminami}, distinguishing the ideal on the \emph{alphabet}
from $\II_\square$ on the \emph{branch space}. Put
$B_N=\omega^2\setminus E_N$ and
\[
 J=\{A\subseteq\omega^2:\exists N\ A\subseteq B_N\}.
\]
Then its dual filter
$J^*=\{A\subseteq\omega^2:\omega^2\setminus A\in J\}$
is $\FF_\square$; write
$J^+=\mathcal P(\omega^2)\setminus J$ for the $J$-positive
sets. Let $\operatorname{sep}(J)$ be the least value of
$|\mathcal G|+|\mathcal H|$ for families $\mathcal G\subseteq J$ and
$\mathcal H\subseteq J^+$ for which no $A\subseteq\omega^2$
satisfies $|A\cap G|<\aleph_0$ for every $G\in\mathcal G$ and
$|A\cap H|=\aleph_0$ for every $H\in\mathcal H$.

\begin{lemma}\label{lem:separation}
$\operatorname{sep}(J)\ge\cov(\MM)$.
\end{lemma}
\begin{proof}
Fix $\mathcal G\subseteq J$ and $\mathcal H\subseteq J^+$ with
$|\mathcal G|+|\mathcal H|<\cov(\MM)$.
Let $Y$ be the Polish space of sequences
$y=(a_n)_{n<\omega}\in(\omega^2)^\omega$ with
$\min(a_n)<\min(a_{n+1})$ for all $n$. It is a nonempty closed
subspace of the product of countable discrete spaces. Every finite
increasing sequence has countably infinitely many immediate
extensions, so $Y$ is homeomorphic to Baire space. For
$H\in\mathcal H$ and $k<\omega$, the set
\[
 U_{H,k}=\{y\in Y:\exists n\ge k\ (a_n\in H)\}
\]
is dense open: $H$ has points of arbitrarily large minimum, so
every finite increasing sequence extends through such a point.
Fewer than $\cov(\MM)$ meager subsets cannot cover $Y$.
Hence some $y$ lies in every $U_{H,k}$. Its range $A$ meets each
$H$ infinitely. Since the minima of $a_n$ strictly increase,
$A\cap B_N$ is finite for every $N$, and $A$ meets each
$G\in\mathcal G$ only finitely.
\end{proof}

\begin{lemma}\label{lem:martin-cover}
For the quadrant tree order,
$\mathfrak m(\PP_\square)=\cov(\II_\square)$, where
$\mathfrak m(P)$ is the least number of dense open subsets of $P$
with no common filter.
\end{lemma}
\begin{proof}
For each dense open $D\subseteq\PP_\square$,
$O_D=\bigcup_{T\in D}[T]$ is $\tau_\square$-open dense, so
$X\setminus O_D\in\mathcal N_\square$. If fewer than
$\cov(\II_\square)$ dense sets are given, choose a branch
outside all their complements. Trees containing that branch form
a directed family: for any two, cone their intersection beyond both
stems. Its upward closure is a filter meeting each dense set.
Thus $\mathfrak m(\PP_\square)\ge\cov(\II_\square)$.
Conversely, cover $X$ by $\cov(\II_\square)$ members of
$\II_\square$ and decompose them into grid-null sets. For each
null set $N$ take the dense open set of trees whose bodies miss
$N$, and add the countably many dense sets requiring stem length
at least $n$. A filter meeting all these sets would have a
unique branch outside the cover. No such filter exists.
\end{proof}

\begin{corollary}[Exact covering number]\label{cor:exact-cover}
$\cov(\II_\square)=\add(\MM)$.
\end{corollary}
\begin{proof}
Hru\v{s}\'ak--Minami (Theorem~1.3 of the linked author
preprint \cite{hrusakminami}) prove
$\mathfrak m(L_{J^*})=
\min\{\add(\MM),\operatorname{sep}(J)\}$ whenever $J^*$
is not an ultrafilter. Our $J^*=\FF_\square$ is
not an ultrafilter: the two orientations partition every
quadrant into two $J$-positive sets. Apply the preceding lemmas
and $\add(\MM)\le\cov(\MM)$.
\end{proof}

For an ideal $I$ on a Polish space, define its local covering number
without relying on a variant notation from the literature:
\[
 \cov^{\mathrm{loc}}(I)=
 \min\{|\mathcal A|:\mathcal A\subseteq I,\ \exists B\in
            \operatorname{Bor}(X)\setminus I\ 
            (B\subseteq\bigcup\mathcal A)\}.
\]
For Borel positive $B$, $\cov(I\restriction B)$ is the least
size of a family of restrictions of $I$-sets covering $B$.

\begin{theorem}[Local covering homogeneity]\label{thm:local}
For every Borel $\II_\square$-positive $B\subseteq X$,
\[
 \cov(\II_\square\restriction B)
 =\cov^{\mathrm{loc}}(\II_\square)
 =\cov(\II_\square).
\]
\end{theorem}
\begin{proof}
Pure-quadrant trees are dense. For such a tree $S$ with stem $s$
and threshold map $N(t)$, define $\Phi_S:X\to[S]$ recursively:
write $s$ first and, after an output prefix $t\supseteq s$,
send the next input pair $(i,j)$ to $(N(t)+i,N(t)+j)$.
This is an ordinary homeomorphism and also a homeomorphism
between the grid topology on $X$ and its relative topology on
$[S]$: quadrant successor sets translate to quadrant successor
sets in both directions. The body $[S]$ is grid-open, so its
subspace meager ideal is the restriction of $\II_\square$. Hence
$\cov(\II_\square\restriction[S])=\cov(\II_\square)$.

For Borel positive $B$, the dense idealized presentation gives
$T$ with $[T]\setminus B\in\II_\square$. Choose pure $S\le T$.
If $\kappa$ ideal sets cover $B$, then $\kappa$ is uncountable
because $B$ is ideal-positive and $\II_\square$ is a sigma-ideal.
Their restrictions together with $[S]\setminus B$ cover $[S]$
using $\kappa$ sets, so
$\kappa\ge\cov(\II_\square)$. The reverse inequality follows
by restricting a cover of $X$ to $B$. Taking the minimum over
Borel positive $B$ proves the middle equality.
\end{proof}

\section{Threshold games and strategy transfer}\label{sec:games}

The tail filter gives a game with the same local quantifiers as a
forcing condition. Set $X_\square=(\omega^2)^\omega$.
Fix a finite grid stem $s$ and $X\subseteq X_\square$.
In $G_s(X)$, Builder announces $N_r<\omega$ at
round $r$, and Selector answers $a_r\in E_{N_r}$. Selector wins when
$s^\frown(a_r)_{r<\omega}\in X$. A \emph{pure-quadrant tree} has
successor set exactly $E_{N_t}$ at every node above its stem. A
\emph{positive tree} has successor set meeting every $E_N$ at every
node above its stem. Positive trees need not be conditions of
$\PP_\square$; the diagonal $\{(i,i):i<\omega\}$ is positive and
contains no quadrant. These are the quadrant versions of the
Hechler/Laver game trees in \cite[Proposition~13]{miller}; the filter formulation is
related to \cite{khomskii}.

\begin{theorem}[Exact tree reading of the game]\label{thm:game-trees}
For every $X\subseteq X_\square$ and finite $s$:
\begin{enumerate}
\item Builder wins $G_s(X)$ if and only if a pure-quadrant tree $T$
with stem $s$ satisfies $[T]\cap X=\varnothing$.
\item Selector wins $G_s(X)$ if and only if a positive tree $L$ with
stem $s$ satisfies $[L]\subseteq X$.
\end{enumerate}
For Borel $X$, exactly one alternative occurs.
\end{theorem}
\begin{proof}
A Builder strategy assigns a threshold to each response history;
keep every legal response to that threshold. The resulting tree has
exact quadrant successors and all its branches lose for Selector.
Conversely, the thresholds of a pure-quadrant tree prescribe a
Builder strategy.

Given a Selector strategy, inductively assign one compatible game
history to each node of a tree. At a node carrying a history, collect
the strategy's answers to \emph{all} possible next thresholds.
This set meets every $E_N$, and for each distinct answer retain one
history witnessing it. Every branch of the resulting positive tree
is a play following the strategy, so belongs to $X$. Conversely,
given $L$, answer a threshold $N$ with any element of
$\Succ_L(t)\cap E_N$. This stays in $L$ and wins. For Borel $X$,
Martin's Borel determinacy theorem \cite{martin} supplies a winner;
the two alternatives cannot both hold.
\end{proof}

The same argument applies to the resource game: Builder chooses a
finite $F\subseteq U$, and Selector responds with $a\in X_F$.
In its Selector-tree clause, successor sets must meet every $X_F$.
The following statement compares games whose payoff only sees the
observation. For $Z\subseteq Y^\omega$, write
$\rho^{-\omega}(Z)=\{x\in A^\omega:\rho\circ x\in Z\}$.

\begin{theorem}[Winner transfer through observation]
\label{thm:game-transfer}
Assume \eqref{eq:S} and \eqref{eq:P}, and take a source stem
$s\in A^{<\omega}$ and target stem $\rho\circ s$.
For every $Z\subseteq Y^\omega$, Builder wins the source resource
game with payoff $\rho^{-\omega}(Z)$ if and only if Builder wins the
target threshold game with payoff $Z$. The same equivalence holds
for Selector. No determinacy assumption on $Z$ is needed.
\end{theorem}
\begin{proof}
All four simulations keep the full play history. To transfer a
source Builder strategy to the target, let it choose $F$ and use
(S) to choose $N$ such that each target reply in $E_N^B$ has a
preimage in $X_F$; feed one such preimage to the source strategy.
To transfer a target Builder strategy to the source, use (P) to
choose $F_N$ after the target strategy chooses $N$; every source
reply then projects into $E_N^B$.

For a source Selector strategy to play in the target game, use (P)
to turn the announced $N$ into a source challenge $F_N$, run the
source strategy, and project its reply. For a target Selector
strategy to play in the source game, use (S) to turn the announced
$F$ into $N$, run the target strategy, and lift its reply to $X_F$.
Every simulated play has the same observed payoff, proving all
four implications.
\end{proof}

For $\rho_\square$, this transfers every payoff depending only on
minimum and orientation. The minimum map alone similarly transfers
minimum-only payoffs to the cofinite-tail tree game. The theorem is
about winning strategies. The complete projection of
Theorem~\ref{thm:projection} supplies the separate order-theoretic
lifting statement.

\begin{proposition}[All branches and generic truth differ]
\label{prop:game-forcing}
Let $X_{\rm even}$ be the grid branches whose first coordinate is
eventually even. Every $T\in\PP_\square$ has branches in
$X_{\rm even}$ and in its complement, but
$\PP_\square\Vdash\dot g\notin X_{\rm even}$.
\end{proposition}
\begin{proof}
At every node above the stem, each contained quadrant has cells
of both first-coordinate parities. Recursively choosing even
first coordinates gives a branch in $X_{\rm even}$; choosing odd
ones gives a branch outside it. For each $k$, the set of conditions
whose stem records an odd first coordinate at a stage $n\ge k$
is dense: extend any stem to a sufficiently late such cell and
cone. A generic filter meets these dense sets for all $k$, so its
branch has infinitely many odd first coordinates.
\end{proof}

\section{Statistical shadows and their limit}\label{sec:probability}

Finite cutoffs assign useful exact laws to visible and hidden
coordinates. They are auxiliary sampling choices rather than
measures intrinsic to the forcing. Let $(I,J)$ be uniform on
$Q_{N,n}=[N,N+n-1]^2$, and set
$M=\min(I,J)$, $B=\mathbf1_{I>J}$, and $D=|I-J|$.

\begin{proposition}[Exact square-cutoff law]\label{prop:cutoff}
For $0\le k<n$ let $r=n-k$. The admissible triples
$(M=N+k,B=0,D=d)$ for $0\le d<r$, and
$(M=N+k,B=1,D=d)$ for $1\le d<r$, each have probability
$n^{-2}$. Consequently
\[
 \Pr(M=N+k)=\frac{2r-1}{n^2},\qquad
 \Pr(B=1)=\frac{n-1}{2n},\qquad
 \Pr(D=0)=\frac1n,
\]
and $\Pr(D=d)=2(n-d)/n^2$ for $1\le d<n$. In natural-log units,
\[
 H(D\mid M,B)=\frac1{n^2}\sum_{r=1}^{n}
 \bigl(r\log r+(r-1)\log(r-1)\bigr)
 =\log n-\frac12+o(1),
\]
where $0\log0=0$.
\end{proposition}
\begin{proof}
Each admissible triple reconstructs one cell of the square.
Counting its row and column gives the probabilities. Conditional
on $(M=N+k,B=0)$, the $r$ offsets are uniform; conditional on
$(M=N+k,B=1)$, the $r-1$ offsets are uniform. Weighting their
entropies gives the sum. Its asymptotic follows from the Riemann
sum for $2x\log x$ on $[0,1]$.
\end{proof}

The two orientation components of $((M-N)/n,D/n,B)$ converge to
Lebesgue area measure on the triangle
$\{(u,v):u,v\ge0,\ u+v\le1\}$, each with density one there.
A uniform finite palette $K$ instead has exactly
$H(K\mid M,B)=\log|K|$. These entropy formulas expose a finite
versus growing hidden fiber under this cutoff; the Boolean
density theorem has its own order-theoretic proof.

\begin{proposition}[Exhaustion dependence]\label{prop:exhaustion}
For a uniform cell in $[0,a)\times[0,b)$, with positive integer
$a,b$,
\[
 \Pr(I>J)=
 \begin{cases}
 (a-1)/(2b),&a\le b,\\
 1-(b+1)/(2a),&a\ge b.
 \end{cases}
\]
As $a,b\to\infty$, every fixed quadrant $E_N$ has probability
tending to one, while the limiting orientation probability can be
any value in $[0,1]$. The complement of the diagonal also has
probability tending to one under every such rectangle but contains
no quadrant.
\end{proposition}
\begin{proof}
Count pairs with $i>j$ by summing the available $j$ for each $i$;
the two ranges yield the displayed formulas. Varying $a/b$
obtains all orientation limits. The diagonal has at most
$\min(a,b)$ cells out of $ab$, while its complement omits
$(i,i)$ from every $E_N$.
\end{proof}

The visible pushforward of a square cutoff has the triangular
minimum law of Proposition~\ref{prop:cutoff}; a direct uniform
cutoff on visible $(m,b)$ symbols has a uniform minimum law.
Thus complete projection need not preserve a cutoff distribution.

\begin{theorem}[No fixed probabilistic test]\label{thm:no-walk}
For any history-dependent countably additive transition kernels
$p_t$ on $\omega^2$, there is a pure-quadrant condition $T$ whose
branch body has walk probability zero. More generally, for every
ground-model Borel probability $\mu$ on $X_\square$, there is
$f\in\omega^\omega\cap V$ for which
\[
 \mu\{g:f\leq^*m_g\}=0,
 \qquad \PP_\square\Vdash f\leq^*m_{\dot g}.
\]
\end{theorem}
\begin{proof}
At every history $t$, continuity from above gives $N_t$ with
$p_t(E_{N_t})\le1/2$. The tree using exactly these successor
quadrants is a condition. The walk's chance to remain inside it
through $k$ stages is at most $2^{-k}$, so its body has measure
zero. A walk chosen after $T$ can, of course, be supported on
$[T]$.

For the second assertion, choose $f(n)$ so large that
$\mu\{g:m_g(n)\ge f(n)\}\le2^{-n}$. By Borel--Cantelli,
almost every $g$ meets this event only finitely often, hence does
not eventually dominate $f$. Raising the quadrant at level $n$
above $f(n)$ below any condition is a dense refinement, so the
generic minimum dominates $f$.
\end{proof}

For a second concrete contrast, dense stem extensions put
diagonal symbols into the generic branch infinitely often.
Independent square-cutoff samples of side lengths $L_n$ with
$\sum_n1/L_n<\infty$ hit the diagonal only finitely often almost
surely. Finite statistical shadows can suggest certificate
patterns, but neither probability one nor vanishing failure
frequency supplies the dense-refinement argument below every
condition.

\section{Effective content and degrees}\label{sec:effective}

Fix computable codings of grid cells and stage markers. For a
generic branch $g(n)=(i_n,j_n)$, write
$h(n)=(m_n,b_n)=\rho_\square(g(n))$ and $d(n)=|i_n-j_n|$.
Equation~\eqref{eq:reconstruct} gives
$g\equiv_T h\oplus d$. The characteristic function $x_g$ of the
unique chosen marker at each stage satisfies $x_g\equiv_Tg$:
searching stage markers in $x_g$ finds that unique marker, while
$g$ computes the characteristic function.

\begin{theorem}[Offset Cohen projection and degrees]
\label{thm:offset-degree}
The map taking a grid tree to the offset string of its stem is a
complete projection from $\PP_\square$ to $\omega^{<\omega}$.
Consequently $d$ is a Cohen real over $V$. The reals $h$ and $d$
are Turing-incomparable, and
\[
 x_g\equiv_Tg\equiv_T h\oplus d.
\]
Nevertheless $\PP_\square$ forces
$d(n)\le m(n+1)$ for all sufficiently large $n$; thus $d$ is
eventually bounded by a real in $V[h]$ and is not Baire Cohen
over $V[h]$.
\end{theorem}
\begin{proof}
Every quadrant $E_N$ contains a cell of each prescribed offset
$k$, for example $(N+k,N)$. Hence the offset projection of a
condition contains every finite offset extension of its stem.
Given a longer Cohen string, lift it successively to a node of
the grid tree and cone there. This proves the complete-projection
lifting property, so $d$ is Cohen over $V$.

The minimum $m$ dominates every ground function. Cohen forcing
adds no real dominating every ground function: for a name
$\dot f$, enumerate pairs $(s,k)$ of Cohen strings and lower
bounds. Choose a distinct $n\ge k$ and an extension of $s$
deciding $\dot f(n)=v$, then set a ground $q(n)>v$. For each $k$,
conditions forcing $\dot f(n)<q(n)$ for some $n\ge k$ are dense.
Therefore $m\notin V[d]$, and $h\not\le_Td$. Conversely the
finite residue $d\bmod2$ is Cohen over $V[h]$ by the commuting
finite-label projection in Corollary~\ref{cor:noncohen}; in
particular $d\notin V[h]$ and $d\not\le_Th$.

At any node, thin the next successor quadrant so that its minimum
exceeds the offset of the preceding cell. This gives a dense set
of conditions enforcing $d(n)\le m(n+1)$ at all stages after
their stems. An eventually bounded real cannot be Baire Cohen
over a model containing its bound, because the set of reals
eventually below that bound is meager there.
\end{proof}

For finite withdrawal $F$, survival of a specified grid cell in
$X_F$ is decidable and a quadrant threshold can be computed from
$F$. Relative to an oracle for arbitrary $X\subseteq\omega^2$,
membership $X\in\FF_\square$ has the $\Sigma^0_2(X)$ form
$\exists N\,\forall i,j\ge N\ ((i,j)\in X)$. For a valid
computable tree, a local threshold can be found with $0'$;
deciding that an arbitrary tree is valid is a different task.
The uncountable Boolean density excludes a computably enumerated
dense set of conditions. An abstract support incidence therefore
supplies no uniform Turing machine: an implementation needs
computable symbol, support, observation, guard, schedule, producer,
and verifier procedures.

\section{Uniform certificates and arithmetic applications}\label{sec:certificates}

Fix a nonempty countable set $\mathcal C$ of task codes, distinct
from the resource-token space $U$, and a schedule $(u_n)_{n<\omega}$
in which every code recurs infinitely often. Let $C_n\subseteq A$ be
the instructions whose finite readout at stage $n$ is accepted by
a fixed verifier for the code $u_n$. All these objects belong to
the ground model. The verifier is \emph{sound} if acceptance of a
readout for $u$ implies the target formula $\varphi_u$, from premises
whose truth has independently been established. This section first
studies the sets $C_n$ as combinatorial objects; soundness supplies
their interpretation as proofs.

For the filter $\FF$ generated by the survivor sets $X_F$, a set
$C$ is \emph{positive} if $C\cap X_F\ne\varnothing$ for every
finite withdrawal $F$. It is \emph{filter-large} if it contains
some $X_F$, equivalently $C\in\FF$. These are different hypotheses.
Define
\[
 X_{\rm cert}=\bigcap_u\bigcup_{n:u_n=u}
                  \{g\in A^\omega:g(n)\in C_n\}.
\]
This is a $G_\delta$ payoff in the product topology of the discrete
alphabet $A$. Every branch in $X_{\rm cert}$ reads a certificate
for each scheduled code.

\begin{theorem}[Positive certificates give dense success]
\label{thm:certificates}
Suppose that for each $u$ there are arbitrarily late stages $n$
with $u_n=u$ such that $C_n$ is positive.
Then the set $D_u$ of conditions whose stem includes an accepted
$u$-cell at one of those stages is open dense. Every generic
branch belongs to $X_{\rm cert}$. In the grid instance,
positivity means $C_n\cap E_N\ne\varnothing$ for every $N$.
\end{theorem}
\begin{proof}
Fix a condition $T$ and choose a qualifying stage $n$ beyond its
stem. Extend inside $T$ to a node of length $n$. Its successor set
contains $X_F$ for some finite $F$; choose an accepted cell in
$C_n\cap X_F$ and cone at the resulting node. This strengthening
belongs to $D_u$. Accepted stem cells persist under strengthening.
Genericity meets every ground-model dense set $D_u$.
\end{proof}

In the quadrant threshold game, the same positivity hypothesis gives
Selector a winning strategy for $X_{\rm cert}$ from every stem:
process the codes in order, wait for a positive occurrence of the
current code, answer Builder's threshold $N$ with a cell in
$C_n\cap E_N$, and proceed to the next code. Theorem~\ref{thm:game-trees}
then gives a positive tree contained in $X_{\rm cert}$. For a general
Borel payoff that theorem gives a dichotomy by Borel determinacy
\cite{martin}; here positivity supplies an explicit winning strategy.
A positive tree need not be a forcing condition.

\begin{theorem}[Filter-large certificates give branchwise refinement]
\label{thm:branchwise}
Let $\FF$ be a free filter on a countably infinite alphabet $A$,
let $\PP=L_{\FF}(A)$, and put $I=\{n:C_n\in\FF\}$.
Here free means that $\FF$ contains every cofinite set and no
finite set.
Suppose that $\{n\in I:u_n=u\}$ is unbounded in $\omega$ for
every code $u$. Then
\[
 \mathcal D_{\rm cert}=\{T\in\PP:[T]\subseteq X_{\rm cert}\}
 \quad\text{is dense in }\PP.
\]
More precisely, every $T\in\PP$ has a strengthening $T^*$ with
the same stem such that every $g\in[T^*]$ satisfies
$g(n)\in C_n$ whenever $n\in I$ and $n\ge|\stem(T)|$.
If $I=\omega$, every such branch is accepted at every stage
beyond the stem.
\end{theorem}
\begin{proof}
Keep the stem and recursively restrict successors at every
retained node $t$ above it by
\[
 \Succ_{T^*}(t)=
 \begin{cases}
 \Succ_T(t)\cap C_{|t|},& |t|\in I,\\
 \Succ_T(t),&|t|\notin I.
 \end{cases}
\]
At a restricted node both intersectands belong to $\FF$, so their
intersection does too. Thus $T^*$ is a condition below $T$ with
the same stem. Each code has a qualifying occurrence beyond this
stem, so every branch of $T^*$ lies in $X_{\rm cert}$.
\end{proof}

In the resource presentation the intersection contains
$X_{F\cup F'}$ whenever $X_F\subseteq\Succ_T(t)$ and
$X_{F'}\subseteq C_{|t|}$. In the grid presentation it contains
$E_{\max(M,N)}$ whenever the two sets contain $E_M$ and $E_N$.
The theorem therefore proves simultaneous certificate success on
all branches by direct recursive thinning.
Its stronger premise is essential.
The support-generated filter $\FF_s$ is free by the argument
preceding Lemma~\ref{lem:rank}, so this theorem applies to the
resource game above.

\begin{proposition}[Positivity alone does not give branchwise refinement]
\label{prop:positive-not-uniform}
There are a recurrent schedule of countably infinitely many codes
and positive sets $C_n\subseteq\omega^2$ such that every $D_u$ is
open dense and Selector wins $G_s(X_{\rm cert})$ from every stem,
but no $T\in\PP_\square$ satisfies $[T]\subseteq X_{\rm cert}$.
\end{proposition}
\begin{proof}
Take any recurrent schedule of all natural numbers and put
$C_n=\{(i,i):i<\omega\}$ for every $n$. Each $C_n$ meets every
quadrant, so Theorem~\ref{thm:certificates} applies. Selector may
always answer a threshold $N$ with $(N,N)$ and therefore wins.
Fix a condition $T$. Only finitely many codes have an accepted
occurrence in its stem; choose another code $u$. At every later
occurrence of $u$, choose an off-diagonal successor, available in
every final quadrant. At other stages choose any successor.
The resulting branch of $T$ has no accepted occurrence of $u$,
so it is outside $X_{\rm cert}$. The acceptance predicate can
certify the tautology $u=u$, so this separation is compatible
with verifier soundness.
\end{proof}

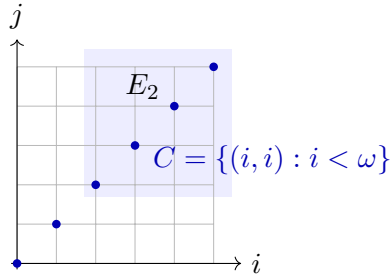
\begin{center}
\begin{tikzpicture}[scale=.52]
\fill[blue!7] (1.7,1.7) rectangle (5.45,5.45);
\draw[step=1,gray!55,very thin] (0,0) grid (5,5);
\draw[->] (-.15,0) -- (5.7,0) node[right] {$i$};
\draw[->] (0,-.15) -- (0,5.7) node[above] {$j$};
\foreach \i in {0,...,5} {
  \fill[blue!70!black] (\i,\i) circle (3.1pt);
}
\node[anchor=west] at (2.5,4.5) {$E_2$};
\node[anchor=west,blue!70!black] at (3.2,2.65)
 {$C=\{(i,i):i<\omega\}$};
\end{tikzpicture}
\end{center}
This finite window continues in both coordinate directions.
The diagonal $C$ meets every final quadrant $E_N$ but contains
none of them. That distinction is why a positive strategy tree
need not be a full-quadrant forcing condition.

Thus three objects must be distinguished: a generic branch meeting
all dense certificate requirements, a positive tree describing a
winning Selector strategy, and a forcing condition whose every
branch succeeds. The first two follow from positive acceptance;
Theorem~\ref{thm:branchwise} supplies the third from filter-large
acceptance.

The precise condition for dense branchwise refinement is elementary
when acceptance depends only on the current stage and symbol.

\begin{proposition}[Branchwise density criterion]
\label{prop:branchwise-criterion}
Let $\FF$ be a free filter on a countable alphabet $A$, let every
code occur infinitely often in $(u_n)$, and let $C_n\subseteq A$.
Put
\[
 L(u)=\{n:u_n=u,\ C_n\in\FF\},\qquad
 P(u)=\{n:u_n=u,\ C_n\text{ is }\FF\text{-positive}\}.
\]
The family $\{T\in L_{\FF}(A):[T]\subseteq X_{\rm cert}\}$ is dense
if and only if
\begin{enumerate}
\item only finitely many codes $u$ have $L(u)=\varnothing$;
\item for each code $u$, either $P(u)$ is unbounded or $C_n=A$
at some occurrence of $u$.
\end{enumerate}
\end{proposition}
\begin{proof}
Suppose first that infinitely many codes have no filter-large
occurrence. For any condition choose one of these codes not accepted
in its finite stem. At every later occurrence of that code, the
successor set belongs to $\FF$ but is not contained in $C_n$;
choose a successor outside $C_n$. This produces a branch with no
certificate for the code, so no condition is branchwise successful.

Next suppose $P(u)$ is bounded and no occurrence of $u$ has $C_n=A$.
Choose $l$ beyond every positive occurrence. Since each earlier
$u$-stage has a proper accepted set, choose a stem of length $l$
avoiding all those sets. At each later $u$-stage, $A\setminus C_n$
belongs to $\FF$ by nonpositivity. A tree with that stem and those
successor sets at later $u$-stages has no branch accepting $u$;
no strengthening can be branchwise successful.

Conversely, take a condition $T$ with stem $s$ of length $l$.
Intersect its successors with $C_n$ at every stage $n\ge l$ for
which $C_n\in\FF$, giving a same-stem condition $T^0$.
Let $E$ be the codes not accepted in $s$ with no filter-large
occurrence at or after $l$. This set is finite: such a code either
has no filter-large occurrence or has one before $l$. No code in
$E$ has an occurrence with $C_n=A$, since such a stage would accept
it in $s$ or be filter-large after $l$. Thus $P(u)$ is unbounded
for each $u\in E$. Successively extend the stem inside $T^0$ to
meet a positive accepted set at an occurrence of each code in $E$.
The final cone accepts those finitely many codes in its stem; every
other code is accepted in $s$ or at an enforced filter-large stage.
Every branch therefore belongs to $X_{\rm cert}$.
\end{proof}

\begin{corollary}[Effective tail bounds and visible sufficient payoffs]
\label{cor:effective-certificates}
Suppose in the grid case that a computable function $n\mapsto N_n$
satisfies $E_{N_n}\subseteq C_n$ for every $n$. Then there is a
computable pure-quadrant tree with every branch in $X_{\rm cert}$,
and Selector has a computable winning strategy. The visible payoff
\[
 Y_N=\{h\in(\omega\times2)^\omega:
          \exists r\ \forall n\ge r\ h(n)_0\ge N_n\}
\]
satisfies $\rho_\square^{-\omega}(Y_N)\subseteq X_{\rm cert}$.
\end{corollary}
\begin{proof}
Take the tree whose successors at level $n$ are exactly $E_{N_n}$.
Membership is decided by finitely many computable inequalities.
Against Builder's threshold $M$ at stage $n$, Selector chooses
$(K,K)$ with $K=\max(M,N_n)$. Every later stage is accepted.
An eventually adequate visible minimum gives accepted cells at
all sufficiently late stages, and every code recurs infinitely.
This proves the inclusion.
\end{proof}

The observation-transfer theorem applies to $Y_N$ and its pullback.
It need not apply directly to $X_{\rm cert}$: acceptance may inspect
offset parity, offset magnitude, or two different time budgets.
A computable threshold sequence gives a computable witness tree;
it does not make refinement inside an arbitrary noncomputable
condition effective.

\subsection*{A method for universally quantified finite problems}
Let $R(u)$ be an arithmetic property of a finite code. A typical
example is
\[
 R(u)\ \equiv\ \operatorname{Valid}(u)\Longrightarrow F(u)\ge0,
\]
where validity and the integer-valued function $F$ are computable.
This describes inequalities or feasibility bounds for finite
combinatorial structures. Specify a decidable predicate
$\operatorname{Check}(u,w)$ for finite certificates and prove its
soundness
\begin{equation}\label{eq:checker-soundness}
 \forall u,w\ \bigl(\operatorname{Check}(u,w)\Longrightarrow R(u)\bigr).
\end{equation}
A cell is accepted only when its finite computation produces such
a $w$. The next implication does not assume the target conclusion.

\begin{proposition}[Sound strategies prove the target]
\label{prop:strategy-target}
Assume \eqref{eq:checker-soundness}. If Selector has a winning
strategy in the grid threshold game for the simultaneous
certificate payoff, then
$\forall u\,R(u)$. The same conclusion follows from a condition
$T$ with $[T]\subseteq X_{\rm cert}$.
\end{proposition}
\begin{proof}
Play the grid game against the fixed Builder sequence of zero
thresholds, or
choose a branch of $T$. For each $u$, the resulting branch contains
an accepted readout $w$. Equation~\eqref{eq:checker-soundness}
yields $R(u)$. This reasoning takes place in the ground model.
\end{proof}

A possible proof route is therefore to establish soundness by a
finite local argument, and then prove uniform availability of
accepted instructions by a combinatorial construction, a termination
argument, or an explicit winning strategy. Positivity gives dense
generic success; filter-large availability also gives branchwise
refinement. These existence assertions carry the unresolved
mathematical content. Defining acceptance to mean that $R(u)$ is
true supplies no proof of availability.

For example, suppose finitely many instruction types have total
certificate producers and terminating verifiers. If their finite
runs for code $u_n$ are accepted, choose $N_n$ above all producer
and verifier run times. Then every grid cell in $E_{N_n}$ has
adequate coordinate budgets, whichever type it selects. If totality
and acceptance are proved uniformly, simulation computes such an
$N_n$, and Corollary~\ref{cor:effective-certificates} applies.
The substantive task is that uniform totality-and-acceptance proof;
statistical success rates cannot replace it.

\begin{proposition}[Absoluteness boundary]\label{prop:absoluteness}
For an arithmetic sentence $\varphi$ with natural-number
parameters and any set forcing $P$ with a largest condition over a transitive $V$,
$V\models\varphi$ if and only if
$V\models1_P\Vdash\varphi$. A certificate scheme for a
set-theoretic sentence $\Phi$ transfers an extension-side result
back to $V$ only with a separately proved downward-absoluteness
principle for $\Phi$ in the relevant class of extensions.
\end{proposition}
\begin{proof}
Set forcing preserves $\omega$ and the standard operations and
relations on it. Induction on arithmetic formulas therefore
gives the same truth value in $V$ and every $V[G]$, and the
forcing theorem gives the displayed equivalence. No analogous
induction covers unrestricted quantification over sets; a
separate transfer theorem is required there.
\end{proof}

The general workflow is thus a precise sequence of obligations:
give a sound verifier and true premises, prove cofinal positive
accepted cells or a winning strategy, and determine whether filter-large
acceptance also yields branchwise refinement. For an extension-side
proof, identify the absoluteness principle for the target statement. A direct proof may discharge the target before forcing
enters; the forcing framework records uniformity across all
conditions and the informational structure of the resulting oracle.

\section{Oracle interpretation}

The forcing statements above depend on support incidence and the
tree order. To interpret the branch as an oracle-machine instruction
sequence, one must provide a uniform machine that decodes each
alphabet symbol, tests each guard, and reads a code for the associated
complete conditional derivation. For a concrete implementation,
fix a computable coding of cell markers $e_{n,i,j}$ and let $x_g$
be the characteristic function of exactly the markers
$e_{n,g(n)_0,g(n)_1}$, one at each stage. Searching the stage-$n$
markers in $x_g$ computes $g(n)$, and $g$ computes $x_g$; hence
$x_g\equiv_T g\equiv_T m\oplus b\oplus d$ by
\eqref{eq:reconstruct}. This is a coding statement about generic
objects. It does not assert that the markers already carry sound
derivations, that premises become true by forcing, or that an
arbitrary abstract support assignment is a machine implementation.

The quotient order is $\sigma$-centered in $V[H]$. To see this,
put
$D_h=\{d\in X_h:\exists N\ \forall n\ge N\ d(n)=b(n)\}$.
This set is countable. Every nonempty ground-coded fiber meets it:
below any visible condition, lift its stem to a ground grid node,
then thin visible successor tails so that every later symbol uses
its canonical grid lift $(M,M)$ or $(M+1,M)$. Genericity gives a
canonical-offset branch in the fiber. The same construction after
coning at a finite offset prefix shows $D_h\cap K_T(h)$ is
ordinary-dense in $K_T(h)$. For each $d\in D_h$, the quotient
conditions whose fibers contain $d$ form a centered family by
Proposition~\ref{prop:fiber}; these countably many families cover
$R_H$. The countable point set does not make the Boolean density
countable: Theorem~\ref{thm:density} computes it as $\mathfrak d^V$.

Theorem~\ref{thm:ideal}, Corollary~\ref{cor:exact-cover}, and
Theorem~\ref{thm:local} give the ideal and local covering results
used to describe the grid order. Part~II now supplies one concrete
sound readout with a complete arithmetic proof.

\section{Questions}\label{sec:questions}
The following questions are not settled by the present analysis.
\begin{enumerate}
\item What is the exact value of $\non(\II_\square)$ in terms of
standard cardinal characteristics?
\item Does the full offset quotient $R_H$ add a real dominating
every function of $V[H]$? Is its Boolean completion equivalent
to a familiar named forcing?
\end{enumerate}

\part{Wilf's inequality for thirteen left elements}

\section{The theorem and two reading routes}\label{sec:wilf-intro}
Let $S\subseteq\N=\{0,1,\ldots\}$ be a numerical semigroup: a
cofinite additive submonoid. Its multiplicity $m$ is its least
positive element; its conductor $c$ is the least integer with
$c+\N\subseteq S$. A positive element is \emph{primitive} if it is
not a sum of two positive elements of $S$; otherwise it is
\emph{decomposable}. Write
\[
 P(S)=\{x\in S:x>0\text{ primitive}\},\qquad e=|P(S)|,
 \qquad L=S\cap[0,c),\qquad n=|L|.
\]
The embedding dimension $e$ counts primitives above the conductor
as well. Wilf asked whether $c\le en$ for every numerical semigroup
\cite{wilf}. Eliahou and Mar\'in-Arag\'on proved it for $n\le12$
\cite{eliahou-marin}. Here the finite arithmetic statement is
\begin{theorem}[Thirteen left elements]\label{thm:main}
If $S$ is a numerical semigroup and $|S\cap[0,c)|=13$, then
$c\le13e$.
\end{theorem}

There are two reading routes. Readers interested in forcing can use
Part~I for the order, projections, games and certificate principles,
then read the present application; the appendix proves the arithmetic
premise on which its certificates depend. Numerical-semigroup readers
can start with this section and proceed directly to the arithmetic
appendix. That proof uses the definitions above and its own lemmas,
without forcing or Machine Calculus prerequisites. Its four cases are
$e\le3$, $q\le3$, $q=4$ with $e\ge4$, and $q\ge5$ with $e\ge4$,
where $q=\lceil c/m\rceil$; the first uses a bounded exact
evaluation. Theorem~\ref{thm:main} follows from those arithmetic
cases before any forcing argument. The forcing analysis describes
dependency, persistence, and uniform certificate selection.

The appendix uses Ap\'ery-level counts, graphs of additive
decompositions and a small cubic shadow bound. Eliahou's work
relates Wilf's problem to level counts and Hilbert functions
\cite{eliahou-macaulay}, and to graphs of Ap\'ery sums
\cite{eliahou-graph}. These provide context for the methods;
all lemmas needed for the present arithmetic proof are proved
in the appendix.

\section{Finite codes and proof supports}\label{sec:wilf-support}
Write $C(S)$ for $c(S)\le13e(S)$. Every numerical semigroup has a
finite gap set $F=\N\setminus S$. Fix the coding
$F_u=\{x:\text{the }x\text{-th binary digit of }u\text{ is }1\}$,
and put $S_u=\N\setminus F_u$ for every $u$. Validity is
decidable: with the candidate conductor $c=0$ for
$F_u=\varnothing$ and $c=1+\max F_u$ otherwise, check
$0\notin F_u$ and closure for $x,y<c$ with $x,y\notin F_u$
and $x+y<c$. All larger sums belong to $S_u$ automatically.
For a valid code with $c>0$, find $m$ by bounded search; every
primitive lies below $c+m$, since $z\ge c+m$ decomposes as
$m+(z-m)$. For $S_u=\N$, use $m=e=1$ and $c=n=q=0$.
For invalid codes set the invariant outputs $c,n,e,q,m$ to zero.
Write $\operatorname{NS}(u)$ for the validity test and
$c(u),n(u),e(u),q(u),m(u)$ for these total computable outputs. A
\emph{finite arithmetic trace}
lists membership and primitive flags in these bounded ranges
with their checked counts.

Specialize the finite support kernel of Part~I to the following
arithmetic claims. Write $U_{\rm ar}=\mathcal P\sqcup\mathcal E$
for its finite premise-and-artifact token space. For
$W\subseteq U_{\rm ar}$, a complete derivation is available when
its support lies in $W$; it is sound for $S$ when its premise
formulas hold for $S$. Equation~\eqref{eq:support-equation} gives
the syntactic activation test, with truth checked separately.

Let $J=\{e,q,4,h\}$ and define predicates
\[
 H_e:(e\leq3),\quad H_q:(q\leq3),\quad
 H_4:(e\geq4\land q=4),\quad H_h:(e\geq4\land q\geq5).
\]
For an $S$ with $n=13$, let $v$ state $C(S)$ and $v_j$ state
$H_j(S)\Rightarrow C(S)$. The appendix supplies four checked
conditional derivations $i_j$ of $v_j$ under the premise
$p_n:(n=13)$:
\[
\begin{array}{c|c}
j&\text{arithmetic result}\\ \hline
e&\text{Proposition~\ref{se:main}}\\
q&\text{Theorem~\ref{sd:theorem}}\\
4&\text{Proposition~\ref{qf:main}}\\
h&\text{Theorem~\ref{hd:theorem}}
\end{array}
\]
The artifact $a_j$ applies $v_j$ using $p_j:H_j(S)$, and $j_*$
joins the four implications by the exhaustive case split. The
same finite diagram applies uniformly to every $S$; only the
truth values of its guards change.

\begin{proposition}[Support persistence]\label{prop:persistence}
The minimal complete supports for $v$ are
\[
 G=\{p_n,j_*,i_e,i_q,i_4,i_h\},\qquad
 B_j=\{p_n,i_j,p_j,a_j\}\quad(j\in J).
\]
For $J(S)=\{j\in J:H_j(S)\}$ and available resources
$W\subseteq U_{\rm ar}$, a registered sound derivation of $C(S)$
in this kernel survives exactly when
\begin{equation}\label{eq:persistence}
 [G\subseteq W]\ \lor\ \bigvee_{j\in J(S)}[B_j\subseteq W].
\end{equation}
Equivalently, a withdrawal $D=U_{\rm ar}\setminus W$ destroys
all registered sound derivations exactly when $D$ meets $G$ and every $B_j$
with $j\in J(S)$.
\end{proposition}
\begin{proof}
Every derivation of $v$ ends with $j_*$ or one $a_j$. The first
choice requires all four $i_j$ and $p_n$, giving $G$; the second
requires $i_j,p_n,p_j,a_j$, giving $B_j$. There are no other
artifacts with conclusion $v$. The global case split uses only
the true premise $n=13$; a local application uses a true premise
exactly when $j\in J(S)$. Equation~\eqref{eq:persistence}
follows, and negating it gives the hitting condition.
\end{proof}

The fixed tokens in $G$ and $B_j$ describe semantic proof
availability. Stage-and-cell markers below are fresh copies of
instructions. Repeating one fixed finite support at every cell
would allow a single withdrawal to delete an entire quadrant.

\section{Quadrant readout and uniform Wilf certificates}
\label{sec:wilf-certificate}
Use the grid order $\PP_\square$ of Part~I. Throughout the full
two-label construction, set $W=U_{\rm ar}$, so both the global and
every true local proof template are available. Semantic withdrawals
will be considered after the certificate theorem. At stage $n$, write
$z_{n,i,j}$ for the marker $e_{n,i,j}$ in the stage support of
Section~\ref{sec:quadrant}; thus
$s_{n,\square}(i,j)=\{r_{n,\min(i,j)},z_{n,i,j}\}$.
Its finite-withdrawal filter is $\FF_\square$. The visible
observation is $\rho_\square(i,j)=(\min(i,j),\mathbf1_{i>j})$.
The parity readout
\[
 \theta(i,j)=\bigl(\min(i,j),\mathbf1_{i>j},|i-j|\bmod2\bigr)
\]
is the $k=1$ residue projection of Corollary~\ref{cor:noncohen}.
It factors through the finite-label order of
Theorem~\ref{thm:phase}; over the visible extension its label
factor is Cohen. This is a forcing statement about proof choices,
not an arithmetic derivation.

Fix a computable pairing bijection
$\langle\cdot,\cdot\rangle:\N^2\to\N$ and set
$u_{\langle u,t\rangle}=u$, so every code recurs infinitely often.
At stage $n$ the query is
\[
 R_{13}(u_n):\quad
 \neg\bigl(\operatorname{NS}(u_n)\land n(u_n)=13\bigr)
 \quad\lor\quad c(u_n)\le13e(u_n).
\]
For a valid thirteen-left-element code, let $j(u_n)$ be the first
true predicate among $H_e,H_q,H_4,H_h$. Parity $0$ requests the
global derivation with support $G$; parity $1$ requests the local
derivation with support $B_{j(u_n)}$. Outside the antecedent, both
parities request a finite certificate of that fact. The two
instructions remain distinct even when their conclusions coincide.

A cell $(i,j)$ runs the requested arithmetic-trace producer for
at most $i$ steps. If the producer finishes, its output is passed
to the fixed verifier, which starts a fresh budget of $j$ steps.
A timeout at either step rejects the cell. The producer computes the finite code predicates,
$c,n,e,q$, and their bounded membership and primitive tables.
For a valid code with $n(u_n)=13$, it attaches the appropriate
appendix proof template.
The verifier checks the finite calculations and the hypotheses of
the selected fixed template. It accepts a certified antecedent failure
or a completed template application. Each template is bound to its
proved appendix argument; this supplies verifier soundness.
All these operations are computable. For each fixed $u$ and parity,
if $R_{13}(u)$ holds, the producer and verifier terminate with an
accepted certificate; if $R_{13}(u)$ fails, verifier soundness prevents
acceptance. The first assertion uses the appendix's four cases and
finite evaluator, not genericity. Let $C_n\subseteq\omega^2$ be
the accepted cells at stage $n$. This set is fixed in the ground
model by the producer and verifier.

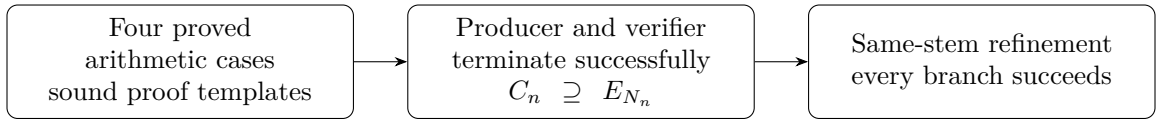
\begin{center}
\begin{tikzpicture}[>=Stealth,
 box/.style={draw,rounded corners,align=center,text width=43mm,
 minimum height=15mm,inner sep=4pt,font=\small}]
\node[box] (arith) at (0,0) {Four proved arithmetic cases\\sound proof templates};
\node[box] (run) at (5.3,0) {Producer and verifier\\terminate successfully\\$C_n\supseteq E_{N_n}$};
\node[box] (tree) at (10.6,0) {Same-stem refinement\\every branch succeeds};
\draw[->] (arith) -- (run);
\draw[->] (run) -- (tree);
\end{tikzpicture}
\end{center}
The first arrow supplies arithmetic truth and termination; the
second organizes those already sound certificates along branches.

\begin{lemma}[Cofinal certificate cells]\label{lem:cofinal}
For each $n$, the following are equivalent:
\begin{enumerate}
\item $R_{13}(u_n)$ holds;
\item there is $N_n$ such that every $(i,j)\in E_{N_n}$ gives an
accepted certificate for its parity-selected instruction;
\item some cell gives an accepted certificate.
\end{enumerate}
\end{lemma}
\begin{proof}
If $R_{13}(u_n)$ holds, both requested producers and verifiers
terminate. Choose $N_n$ above their finitely many run times.
Both coordinate budgets then suffice for either parity, proving
(2). The implication (2)$\Rightarrow$(3) is immediate, and
verifier soundness gives (3)$\Rightarrow$(1).
\end{proof}

Write
\begin{equation}\label{eq:accepted-payoff}
 X_{\rm acc}=\bigcap_{u\in\N}\bigcup_{n:u_n=u}
 \{g\in(\omega^2)^\omega:g(n)\in C_n\}.
\end{equation}
This is a $G_\delta$ payoff. A branch belongs to it precisely
when it reads an accepted certificate for every finite-gap code.

\begin{theorem}[Branchwise Wilf certificates]\label{thm:wilf-branchwise}
Assume the four arithmetic implications proved in the appendix.
Then every $T\in\PP_\square$ has a same-stem refinement $T^*\le T$
such that every $g\in[T^*]$ is accepted at every stage beyond
the stem. In particular,
\[
 \{T\in\PP_\square:[T]\subseteq X_{\rm acc}\}
 \quad\text{is dense in }\PP_\square.
\]
There is a computable pure-quadrant tree with body contained in
$X_{\rm acc}$, and Selector has a computable winning strategy
for this payoff from every finite stem.
For each code $u$, the conditions whose stems include an accepted
$u$-cell form an open dense set $D_u$. Thus every generic branch
reads a sound certificate of $R_{13}(u)$ for every code, and
\[
 1_{\PP_\square}\Vdash
 \forall S\,(S\text{ a numerical semigroup}\land n(S)=13
 \Rightarrow c(S)\le13e(S)).
\]
\end{theorem}
\begin{proof}
The appendix proves every $R_{13}(u_n)$. By Lemma~\ref{lem:cofinal},
for every stage $n$ there is $N_n$ with $E_{N_n}\subseteq C_n$.
Thus $C_n\in\FF_\square$ at all stages. The general branchwise
refinement theorem, Theorem~\ref{thm:branchwise}, now supplies
below every $T$ a same-stem strengthening $T^*$ whose every
branch is accepted at every stage beyond its stem. It also gives
$[T^*]\subseteq X_{\rm acc}$. The positive-certificate theorem,
Theorem~\ref{thm:certificates}, gives each $D_u$ open dense.

Against Builder's threshold $M_r$ at stage $n=|s|+r$,
Selector chooses $(K,K)$ with $K=\max(M_r,N_n)$. Every resulting
branch lies in $X_{\rm acc}$, proving the game claim directly.
The accepting bound $N_n$ can be computed by running the two
terminating producer-verifier pairs for $u_n$ to completion;
Corollary~\ref{cor:effective-certificates} therefore also gives
a computable pure-quadrant certificate tree.

Genericity meets each $D_u\in V$ and verifier soundness gives
$R_{13}(u)$ for each old code. Set forcing adds no natural-number
codes; each numerical semigroup in an extension still has a
finite gap code. The forcing theorem gives the displayed
universal sentence.
\end{proof}

The direct proof of each $R_{13}(u)$ is already supplied by the
appendix and the accepted finite readout. Genericity packages
these certificates on one branch; branchwise refinement is
stronger because \emph{every} branch of a refined condition
succeeds. The accepting tail inclusion, rather than a cutoff
probability, is the decisive premise. The payoff $X_{\rm acc}$
may inspect offset parity and separate coordinate time budgets,
so the observation winner-transfer theorem need not apply to it
directly. The visible sufficient payoff of
Corollary~\ref{cor:effective-certificates} does apply. The visible
minimum is a sufficient lower bound on both execution budgets,
while offset parity selects the proof template. Thus a visible
observation can guarantee successful execution while forgetting
which proof was chosen.

The local tokens $z_{n,i,j}$ can carry distinct copies of each
instruction and its trace. Withdrawing a semantic proof resource
from $U_{\rm ar}$ is different from withdrawing stage markers:
Equation~\eqref{eq:persistence} determines which global or local
templates remain soundly available. The full two-label projection
uses all semantic resources; arbitrary semantic withdrawals need
not preserve both proof labels.

\section{Arithmetic absoluteness and oracle information}
\label{sec:wilf-comparison}
Let $W_{13}$ denote the universal sentence of
Theorem~\ref{thm:main}. Its finite-code form is
\begin{equation}\label{eq:pi01}
 \forall u\in\N\ \bigl[\operatorname{NS}(u)\land n(u)=13
 \Rightarrow c(u)\le13e(u)\bigr].
\end{equation}
The bracketed predicate is decidable by bounded calculations,
so this is a $\Pi^0_1$ arithmetic sentence. By
Proposition~\ref{prop:absoluteness}, for every set forcing $P$
with a largest condition over a transitive $V$,
\[
 V\models W_{13}\quad\Longleftrightarrow\quad
 V\models1_P\Vdash W_{13}.
\]
A finite gap code witnessing failure in $V$ remains a
counterexample in every generic extension. The forcing theorem
in Theorem~\ref{thm:wilf-branchwise} therefore presents an
independently proved arithmetic result; it cannot manufacture
the inequality.

For a generic grid branch $g$, let $h=\rho_\square\circ g$ and
$d(n)=|g(n)_0-g(n)_1|$. With computable stage-marker coding,
Theorem~\ref{thm:offset-degree} gives
$x_g\equiv_Tg\equiv_Th\oplus d$, with $h$ and $d$ Turing
incomparable. The proof-choice bits $d\bmod2$ retain less than
the full offset, although their sequence is Cohen over $V[h]$.
These oracle-degree facts do not measure the difficulty of
Wilf's inequality: the finite producer and verifier already
generate both requested proof-label types from a code using
the arithmetic derivations in the appendix.
\clearpage
\appendix
\noindent\textbf{Arithmetic proof and exact finite evaluation.}
The following sections prove all four implications used above.
Their only computer-assisted step is the exhaustive bounded
evaluation of Lemma~\ref{se:finite-check}; its mathematical
reduction, algorithms, loop bounds, and tallies are printed here.
\medskip
\section{Ap\'ery sets and weighted level counts}\label{sec:foundations}

Throughout the proof of Theorem~\ref{thm:main}, $n=13$. Consequently $c>0$:
if $c=0$, then $S=\mathbb N$ and $L$ is empty. Define
\[
 A=\operatorname{Ap}(S,m)=\{x\in S:x-m\notin S\}.
\]
Here $S$ is viewed as a subset of $\mathbb Z$ when testing $x-m\notin S$.
Equivalently, $x\in A$ means that no $y\in S$ satisfies $x=y+m$;
this convention fixes the meaning of subtraction in the Apéry definition.
Let
\[
 D=\{x\in A:x\text{ is decomposable in }S\},\qquad
 \delta=|D|,\qquad T=A\cap(0,c).
\]
Thus $D$ denotes only the decomposable \emph{Ap\'ery} elements, not all
decomposable elements of $S$.

\begin{lemma}\label{lem:apery}
The following properties hold.
\begin{enumerate}
 \item $A$ contains one representative of each residue class modulo $m$;
       in particular $|A|=m$, and $A\subseteq[0,c+m)$.
 \item $P=\{m\}\mathbin{\dot\cup}(P\cap A)$ and
       $A=\{0\}\mathbin{\dot\cup}(P\setminus\{m\})
       \mathbin{\dot\cup}D$. Hence
       \begin{equation}\label{eq:embedding}
                            m=e+\delta.
       \end{equation}
 \item Every element of $S$ is a sum of primitive elements.
 \item If $z=a+b\in A$ with $a,b\in S$ positive, then $a,b\in T$.
 \item Every $x\in S$ has a unique expression $x=a+km$ with
       $a\in A$ and $k\in\mathbb N$.
\end{enumerate}
\end{lemma}
\begin{proof}
Each residue class contains an element of $S$, since all sufficiently large
integers lie in $S$. Its least such element lies in $A$. Conversely, if
$a,b\in A$ are congruent modulo $m$ and $a<b$, then $b=a+km$ for some
$k\ge1$, so $b-m\in S$, a contradiction. If $a\ge c+m$, then
$a-m\ge c$ lies in $S$, again a contradiction. This proves (1).

The multiplicity $m$ is primitive, and $m\notin A$ because $0\in S$.
If a primitive $p\ne m$ were $y+m$ with $y\in S$, then $y>0$,
contradicting primitivity. Thus every other primitive lies in $A$.
Every positive element is either primitive or decomposable, and the two
possibilities are disjoint. This proves (2), including finiteness of $P$.
For (3), use strong induction on $x$: a positive nonprimitive element splits
as $u+v$ with $0<u,v<x$, and the induction hypothesis applies to both.

If $a=y+m$ for some $y\in S$, then $z=(y+b)+m$, contradicting $z\in A$.
Thus $a\in A$, and similarly $b\in A$. Since $b\ge m$ and $z<c+m$,
we obtain $a<c$; similarly $b<c$. This gives (4).
Repeated subtraction of $m$ while the result stays in $S$ gives the
expression in (5). Its uniqueness follows from the uniqueness of the
Ap\'ery representative in its residue class and cancellation of $m$.
\end{proof}

Put
\begin{equation}\label{eq:depth}
 q=\left\lceil\frac cm\right\rceil,\qquad
 \rho=qm-c,\qquad 0\le\rho<m.
\end{equation}
For $x\in\mathbb N$, define its level by
\[
 \lambda(x)=\left\lfloor\frac{x+\rho}{m}\right\rfloor,
 \qquad A_i=\{a\in A:\lambda(a)=i\},\qquad \alpha_i=|A_i|.
\]
Equivalently, $A_i=A\cap[im-\rho,(i+1)m-\rho)$.
We have $A_0=\{0\}$, all positive elements of $S$ have positive level,
and every element of $A$ has level at most $q$. Furthermore,
\begin{equation}\label{eq:level-left}
 x<c\quad\Longleftrightarrow\quad \lambda(x)<q,
 \qquad \lambda(a+km)=\lambda(a)+k.
\end{equation}
For sums one has
\begin{equation}\label{eq:level-add}
 \lambda(a+b)\ge\lambda(a)+\lambda(b)-1.
\end{equation}
Indeed, $a+b+\rho\ge(\lambda(a)+\lambda(b))m-\rho$
and $\rho<m$. The displayed statements also follow directly from the
half-open intervals defining the levels.

\begin{lemma}\label{lem:weights}
The depth and level cardinalities satisfy
\begin{align}
 1\le q&\le13,\label{eq:q-range}\\
 13&=q+\sum_{i=1}^{q-1}(q-i)\alpha_i,
       \label{eq:weighted-profile}\\
 \sum_{x\in T}\bigl(q-\lambda(x)\bigr)&=13-q.
       \label{eq:weight-sum}
\end{align}
Writing $t=|T|$, in particular $t\le13-q$.
If $d_T$ is the number of decomposable elements in $T$, then
\begin{equation}\label{eq:left-primitives}
                         e\ge1+t-d_T.
\end{equation}
\end{lemma}
\begin{proof}
The $q$ distinct multiples $0,m,\ldots,(q-1)m$ all lie below $c$,
so $q\le n=13$. The positivity of $c$ gives $q\ge1$.
By Lemma~\ref{lem:apery}, the elements of $S$ split into disjoint
progressions $a+\mathbb Nm$, $a\in A$. By \eqref{eq:level-left}, exactly
$q-\lambda(a)$ terms of the progression lie below $c$; this number is
nonnegative because $\lambda(a)\le q$. Summing gives
$n=\sum_{a\in A}(q-\lambda(a))$. The term $a=0$ is $q$, and those
with $a\ge c$ contribute zero. This proves the two equalities.
Each term indexed by $T$ is at least one, which proves $t\le13-q$.
Finally, the $t-d_T$ primitive elements in $T$, together with $m\notin A$,
are distinct members of the full primitive set $P$.
\end{proof}

Combining \eqref{eq:embedding} and \eqref{eq:depth} gives the useful
equivalence
\begin{equation}\label{eq:defect}
 c\le13e
 \quad\Longleftrightarrow\quad
 q\delta\le(13-q)e+\rho.
\end{equation}
Much of the argument bounds the number $\delta$ of decomposable Ap\'ery
elements in terms of the short weighted profile
\eqref{eq:weighted-profile}.

\section{Pair counts}\label{sec:pairs}

An unordered pair in this paper is a two-element \emph{multiset}; repeated
entries are permitted and a loop $\{x,x\}$ is counted once. Let
\[
 \mathcal E=\{\{x,y\}:x,y\in T,
                 \ \lambda(x)+\lambda(y)\le q+1\},\qquad M=|\mathcal E|.
\]
These are eligible pairs; their sums need not all belong to $A$.

\begin{lemma}\label{lem:pairs}
The eligible-pair count satisfies
\begin{equation}\label{eq:pair-bounds}
 \delta\le M\le\binom{t+1}{2},\qquad
 (q-1)M\le(t+1)(13-q)\le(14-q)(13-q).
\end{equation}
It is determined by the profile through
\begin{equation}\label{eq:pair-profile}
 M=\sum_{\substack{1\le i<j<q\\i+j\le q+1}}\alpha_i\alpha_j
   +\sum_{\substack{1\le i<q\\2i\le q+1}}
                  \binom{\alpha_i+1}{2}.
\end{equation}
\end{lemma}
\begin{proof}
Every $z\in D$ has a factorization $z=x+y$ with $x,y\in T$ by
Lemma~\ref{lem:apery}. Since $z<c+m=(q+1)m-\rho$,
\[
 (\lambda(x)+\lambda(y))m\le x+y+2\rho
             <(q+1)m+\rho<(q+2)m.
\]
Thus $\lambda(x)+\lambda(y)\le q+1$, so the sum map from eligible pairs
covers $D$. This gives $\delta\le M$, while the total number of unordered
pairs with repetition from $T$ is $\binom{t+1}{2}$.

Assign $w(x)=q-\lambda(x)$ to each $x\in T$. Each eligible pair has
$w(x)+w(y)\ge q-1$. Summing over all unordered pairs with repetition,
each $w(x)$ occurs $t-1$ times with another element and twice in its loop.
Consequently
\[
 (q-1)M\le\sum_{\{x,y\}\in\operatorname{Sym}^2(T)}(w(x)+w(y))
       =(t+1)\sum_{x\in T}w(x)=(t+1)(13-q).
\]
Now use $t\le13-q$. Finally, sorting pairs by their two levels gives
\eqref{eq:pair-profile}: distinct levels contribute products of their
cardinalities, and a single level contributes pairs with repetition.
\end{proof}

In particular, the profile and pair count take the following forms at
depths four, five and six:
\begin{align}
 q=4:\quad&3\alpha_1+2\alpha_2+\alpha_3=9,\notag\\
 &M=\binom{\alpha_1+\alpha_2+1}{2}
          +(\alpha_1+\alpha_2)\alpha_3;\label{eq:pairs-four}\\
 q=5:\quad&4\alpha_1+3\alpha_2+2\alpha_3+\alpha_4=8,\notag\\
 &M=\binom{\alpha_1+\alpha_2+\alpha_3+1}{2}
          +(\alpha_1+\alpha_2)\alpha_4;\label{eq:pairs-five}\\
 q=6:\quad&5\alpha_1+4\alpha_2+3\alpha_3+2\alpha_4+\alpha_5=7,\notag\\
 &M=\binom{\alpha_1+\alpha_2+\alpha_3+1}{2}
          +(\alpha_1+\alpha_2+\alpha_3)\alpha_4
          +(\alpha_1+\alpha_2)\alpha_5.\label{eq:pairs-six}
\end{align}
The explicit bounds needed for the subsequent cases are obtained from these
small nonnegative integer profiles.

\clearpage
\section{Depth four}\label{sec:depth-four}

Throughout this section, except in Lemma~\ref{lem:shadow}, assume
$q=4$, $n=13$, and $e\ge4$. Put
\[
 a=\alpha_1,\qquad b=\alpha_2,\qquad d=\alpha_3.
\]
The weighted profile and the desired inequality become
\begin{equation}\label{qf:profile-target}
 3a+2b+d=9,\qquad 4\delta\le9e+\rho.
\end{equation}
We first treat $a>0$ using the graph of actual Ap\'ery sums.
For $a=0$, a partition of $D$ and a restriction on residues give stronger
bounds than the eligible-pair count alone.

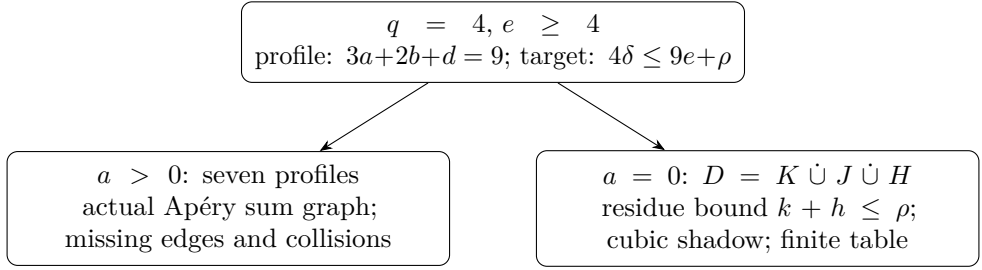
\begin{center}
\begin{tikzpicture}[>=Stealth,
 box/.style={draw,rounded corners,align=center,text width=55mm,
 inner sep=5pt,font=\small}]
\node[box,text width=63mm] (root) at (6.5,2.2)
 {$q=4$, $e\ge4$\\profile: $3a+2b+d=9$; target: $4\delta\le9e+\rho$};
\node[box] (left) at (3,0)
 {$a>0$: seven profiles\\actual Ap\'ery sum graph; missing edges and collisions};
\node[box] (right) at (10,0)
 {$a=0$: $D=K\mathbin{\dot\cup}J\mathbin{\dot\cup}H$\\residue bound $k+h\le\rho$; cubic shadow; finite table};
\draw[->] (root) -- (left);
\draw[->] (root) -- (right);
\end{tikzpicture}
\end{center}
The two branches of this proof use different information about
actual sums; the case tables are the final arithmetic checks after
those structural bounds. In the right branch, $K,J,H$ partition
the decomposable Ap\'ery elements as defined in the empty-first-level
subsection; $k=|K|$ and $h=|H|$.

\subsection{Degrees and collisions of Ap\'ery sums}

Consider the graph with vertex set $T$ whose edges are the unordered pairs
$\{u,v\}$ such that $u+v\in A$. Loops are allowed. Write $E$ for its
edge set and
\[
 N(u)=\{v\in T:u+v\in A\},\qquad
 \deg(u)=|N(u)|,\qquad R=T\cap D,\qquad \kappa=|R|.
\]
A loop contributes one to this degree. Lemma~\ref{lem:apery} shows that
the sum map $E\longrightarrow D$ is surjective. Every actual edge is
eligible, so
\begin{equation}\label{qf:edge-bound}
 \delta\le |E|\le M,
 \qquad e\ge1+|T|-\kappa.
\end{equation}

\begin{lemma}\label{qf:graph}
Every decomposable vertex $u\in R$ satisfies $\deg(u)\le\kappa$.
If some decomposable vertex has degree at least two, then
$\delta\le |E|-1$.
\end{lemma}
\begin{proof}
Write $u=s+t$ with $s,t\in S$ positive. For $v\in N(u)$ we have
$(s+v)+t=u+v\in A$. By Lemma~\ref{lem:apery}, $s+v\in T$;
it is decomposable, so $s+v\in R$. The map $v\mapsto s+v$ is injective,
proving the degree bound.

For the second assertion, suppose the sum map on $E$ is injective.
For any $v\in N(u)$, the pairs
\[
 \{v+s,t\},\qquad \{v+t,s\}
\]
are edges with equal sum. Their endpoints belong to $T$ by
Lemma~\ref{lem:apery}. Equality of the unordered pairs forces $s=t$:
the other matching would give $v=0$. Now the edges
\[
 \{v,u\},\qquad \{s,v+s\}
\]
also have equal sum. Their equality forces $v=s$, since the other
matching would give $s=0$. Thus $N(u)\subseteq\{s\}$.
Consequently a decomposable vertex of degree at least two makes the
surjective sum map noninjective, and its image has at most $|E|-1$
elements.
\end{proof}

Put $U=A_1\cup A_2$ and $V=A_3$. At depth four, the eligible pairs are
exactly the pairs within $U$ and those joining $U$ to $V$. Thus, with
$r=a+b$,
\begin{equation}\label{qf:eligible}
 |U|=r,\quad |V|=d,\quad |T|=r+d,\quad
 M=\binom{r+1}{2}+rd.
\end{equation}
For later use, if all pairs $\{u,v\}$ with $v\in W$ are eligible, then
at least $|W|-\deg(u)$ of these distinct pairs are missing from $E$.
Hence
\begin{equation}\label{qf:missing-star}
 |E|\le M-|W|+\deg(u).
\end{equation}
We may take $W=T$ when $u\in U$, and $W=U$ for any vertex $u$.
For distinct $u,v\in V$, their eligible pairs with $U$ are disjoint,
and therefore
\begin{equation}\label{qf:two-stars}
 |E|\le M-2|U|+\deg(u)+\deg(v).
\end{equation}
These counts remain valid for loops because a fixed vertex determines
one unordered pair for each choice of its other endpoint.

We also use a simple consequence of the multiplicity. If $v\in S$ is
positive, then $v\ge m$, so
\begin{equation}\label{qf:strict-level}
 \lambda(u+v)\ge\lambda(u)+1.
\end{equation}
Thus every factor in a decomposition has smaller level than the sum.
In particular, all elements of $A_1$ are primitive. If $A_1=\{x\}$,
then every decomposable element of $A_2$ is $2x$. Consequently
\begin{equation}\label{qf:singleton-first}
 a=1\quad\Longrightarrow\quad
 R\subseteq\{2x\}\cup A_3,
 \qquad \kappa\le1+d.
\end{equation}

\subsection{The seven profiles with a nonempty first level}

\begin{lemma}\label{qf:nonempty}
If $A_1\ne\varnothing$, then $4\delta\le9e$.
\end{lemma}
\begin{proof}
The nonnegative solutions of $3a+2b+d=9$ with $a\ge1$ give precisely
the following table. For each $a\in\{1,2,3\}$, choose
$0\le b\le\lfloor(9-3a)/2\rfloor$ and put $d=9-3a-2b$;
this describes the complete enumeration.
\[
\begin{array}{c|c|c|c}
 a&b&d&M\\\hline
 3&0&0&6\\
 2&0&3&9\\
 2&1&1&9\\
 1&0&6&7\\
 1&1&4&11\\
 1&2&2&12\\
 1&3&0&10
\end{array}
\]
The first four rows have $\delta\le9$, which suffices because $e\ge4$.
Every row has $M\le12$, so $e\ge6$ also suffices. When $e=5$, only
the row $(1,2,2)$ needs an improvement on $\delta\le M$, since
$4\cdot11\le9\cdot5$. It remains to establish the following four
bounds.

\medskip\noindent
\textit{Profile $(1,3,0)$, with $e=4$: $\delta\le7$.}
Here $|T|=|U|=4$ and $M=10$. By
\eqref{qf:edge-bound} and \eqref{qf:singleton-first},
$1\le\kappa\le1$. Choose $u\in R$. Its degree is at most one, and
\eqref{qf:missing-star}, with $W=T$, gives
$|E|\le10-4+1=7$.

\medskip\noindent
\textit{Profile $(1,2,2)$, with $e=5$: $\delta\le11$.}
Here $|U|=3$, $|V|=2$, $|T|=5$, $M=12$, and
$1\le\kappa\le3$. If $\kappa\le2$, choose any $u\in R$ and apply
\eqref{qf:missing-star} with $W=U$ to obtain
$|E|\le12-3+2=11$.
If $\kappa=3$, some $u\in R$ belongs to $U$, because $|V|=2$.
Its full eligible set of neighbors is $T$, giving
$|E|\le12-5+3=10$.

\medskip\noindent
\textit{Profile $(1,2,2)$, with $e=4$: $\delta\le9$.}
Now $2\le\kappa\le3$.
If $u\in R\cap U$, then
$|E|\le12-5+\deg(u)=7+\deg(u)\le10$.
If $\deg(u)\ge2$, Lemma~\ref{qf:graph} improves this to
$\delta\le9$; if $\deg(u)\le1$, already $\delta\le8$.
If $R\cap U=\varnothing$, then $R=V=\{u,v\}$ and $\kappa=2$.
Both degrees are at most two and \eqref{qf:two-stars} gives
$|E|\le6+\deg(u)+\deg(v)\le10$.
If $\deg(u)=2$, use the collision bound. Otherwise
$|E|\le6+1+2=9$.

\medskip\noindent
\textit{Profile $(1,1,4)$, with $e=4$: $\delta\le9$.}
Write $A_1=\{x\}$ and $A_2=\{y\}$, so $U=\{x,y\}$ and
$|T|=6$. The primitive count implies $\kappa\ge3$.
Every decomposition of an element of $R$ uses two elements of lower
level by \eqref{qf:strict-level}; thus both factors belong to $U$.
It follows that
\[
 R\subseteq Q:=\{2x,x+y,2y\}.
\]
Since $|R|\ge3$ and $|Q|\le3$, equality holds and all three elements
of $Q$ are distinct. The element $x$ is primitive. The element $y$ is
also primitive: otherwise $y=2x$, while $2y\in R$ would imply
$4x<c\le4m$, contradicting $x\ge m$.
Consequently $Q=R\subseteq A_3$. There is a single element $z$ such
that $A_3=Q\mathbin{\dot\cup}\{z\}$.

Every decomposable Ap\'ery element is a sum associated to an eligible
pair, hence lies in $(U+U)\cup(U+A_3)$. Substituting the displayed
description of $A_3$ bounds this set by
\[
 Q\ \cup\ \{3x,2x+y,x+2y,3y\}\ \cup\ \{x+z,y+z\}.
\]
It contains at most $3+4+2=9$ values.

\medskip
The two bounds for $e=4$ that equal nine give $4\delta\le36=9e$;
the bound seven is stronger. The exceptional $e=5$ bound gives
$4\delta\le44<45=9e$. This completes every row and every $e\ge4$.
\end{proof}

\subsection{A numerical cubic shadow bound}

For finite sets of integers, write $B+C=\{b+c:b\in B,c\in C\}$.
The next lemma bounds distinct numerical sums, even when different
formal products represent the same number.

\begin{lemma}\label{lem:shadow}
Let $B\subseteq\mathbb N$ be finite with $|B|\le4$, and let
\[
 K\subseteq B+B,\qquad H\subseteq B+B+B.
\]
Suppose that for every $x,y,z\in B$ such that $x+y+z\in H$,
all three sums $x+y$, $x+z$, and $y+z$ belong to $K$.
If $k=|K|\le3$, then
\[
 |H|\le f(k),\qquad
 f(0)=0,\quad f(1)=1,\quad f(2)=2,\quad f(3)=4.
\]
\end{lemma}
The proof uses the hypothesis for \emph{every} representation of a
numerical cubic sum. Choosing a least-rank monomial for each value
makes quadratic divisors canonical by replacement. Here $\mathcal Q$
denotes the set of least-rank quadratic representatives of $K$:
\[
\begin{aligned}
h\in H&\longmapsto\text{least-rank cubic }M_h\\
&\longmapsto\text{canonical quadratic divisors in }\mathcal Q\\
&\longmapsto\text{loop, edge, and triangle count}.
\end{aligned}
\]
This is the step that controls collisions among distinct formal sums.
\begin{proof}
If $B$ is empty, then $H$ is empty. Otherwise enumerate its distinct
elements as $b_0,\ldots,b_{r-1}$, where $1\le r\le4$, and introduce
formal variables $X_0,\ldots,X_{r-1}$. For a monomial $X^u$ of total
degree two or three define its value and its rank by
\[
 \nu(X^u)=\sum_{i=0}^{r-1}u_i b_i,\qquad
 \operatorname{rank}(X^u)=\sum_{i=0}^{r-1}u_i5^i.
\]
All exponents are at most three, so uniqueness of base-five expansions
shows that distinct monomials have distinct ranks.
For each $v\in K$ choose the quadratic monomial of value $v$ having
least rank. Let $\mathcal Q$ be this set of $k$ monomials.
For each $h\in H$ choose the cubic monomial of value $h$ having least
rank.

Every quadratic divisor of such a chosen cubic has value in $K$ by
the hypothesis. Moreover it is the chosen representative of its value.
Indeed, replacing that divisor by a quadratic monomial of the same
value and smaller rank, and retaining the remaining variable, would
produce a cubic of value $h$ and smaller rank. This contradicts the
choice of the cubic. Hence all quadratic divisors of every chosen
cubic belong to $\mathcal Q$. Cubics chosen for different values of
$h$ are distinct.

It remains to count cubic monomials all of whose quadratic divisors
belong to a fixed set $\mathcal Q$ of at most three quadratics.
Regard $X_i^2$ as a loop at $i$ and $X_iX_j$, $i\ne j$, as an edge.
Let $\ell$ be the number of loops and $s=k-\ell$ the number of ordinary
edges. There are three types of allowable cubic:
\begin{itemize}
 \item $X_i^3$, one for each loop;
 \item $X_i^2X_j$, one for each incidence of an ordinary edge with a
       vertex carrying a loop;
 \item $X_iX_jX_t$ with three distinct indices, one for each triangle
       of ordinary edges.
\end{itemize}
Thus their total number is
\[
 \ell+\sum_{\{i,j\}\text{ an ordinary edge}}
       \bigl(\mathbf1_{X_i^2\in\mathcal Q}
             +\mathbf1_{X_j^2\in\mathcal Q}\bigr)
       +\#\{\text{triangles of ordinary edges}\}.
\]
For $k\le3$, the complete bounds for this expression are as follows;
a dash denotes an impossible value of $\ell$.
\[
\begin{array}{c|rrrr|c}
 &\ell=0&\ell=1&\ell=2&\ell=3&\text{maximum}\\\hline
 k=0&0&-&-&-&0\\
 k=1&0&1&-&-&1\\
 k=2&0&2&2&-&2\\
 k=3&1&3&4&3&4
\end{array}
\]
To verify the table, a triangle requires all three available edges and
therefore occurs only in the entry $(k,\ell)=(3,0)$, where there is at
most one. With exactly one loop, each ordinary edge contributes at
most one incidence; with two loops, the only possible remaining edge
contributes at most two; with three loops there are no ordinary edges.
The rows $k\le2$ follow from the same count. This exhausts all
possibilities and proves the stated bound on $|H|$.
\end{proof}

\subsection{An empty first level: partition and residues}

Assume now that $A_1=\varnothing$, and set $B=A_2$, $C=A_3$.
Then
\begin{equation}\label{qf:empty-profile}
 |B|=b,\quad |C|=d,\quad 2b+d=9,
 \quad 0\le b\le4.
\end{equation}
Define three disjoint sets of decomposable Ap\'ery elements:
\begin{equation}\label{qf:partition-def}
 \begin{gathered}
 K=C\cap D,\qquad J=A_4\cap(B+B),\qquad
 H=(A_4\cap D)\setminus(B+B),\\
 k=|K|,\qquad \ell=|J|,\qquad h=|H|.
 \end{gathered}
\end{equation}

\begin{lemma}\label{qf:partition}
With this notation,
\begin{align}
 D&=K\mathbin{\dot\cup}J\mathbin{\dot\cup}H,
       &\delta&=k+\ell+h,\label{qf:partition-card}\\
 K&\subseteq B+B,&H&\subseteq B+C,\label{qf:sum-inclusions}\\
 k+\ell&\le\binom{b+1}{2},&h&\le bd,\label{qf:basic-counts}\\
 e&\ge1+b+d-k,&0\le k&\le3.\label{qf:primitive-empty}
\end{align}
\end{lemma}
\begin{proof}
Every positive Ap\'ery element has level at least two. If an element of
$A_2$ were decomposable, its factors would belong to $T$ and have
levels at least two, whereas \eqref{eq:level-add} would force the sum
to have level at least three. Thus $B$ consists of primitive elements.
For a decomposable element of $A_3$, the factor levels are at least
two and their sum is at most four; both factors therefore belong to
$B$. For a decomposable element of $A_4$, their level sum is at most
five, so its factor levels are $(2,2)$, $(2,3)$, or $(3,2)$.
This proves the partition and the two sumset inclusions.

The sets $K$ and $J$ are disjoint subsets of $B+B$, whose cardinality
is at most the number $\binom{b+1}{2}$ of unordered pairs from $B$.
The inclusion $H\subseteq B+C$ gives $h\le bd$.
The primitive elements in $T=B\mathbin{\dot\cup}C$ are precisely
$B\cup(C\setminus K)$; adding the primitive $m$ gives
$e\ge1+b+d-k$. Finally $k\le d$ and
$k\le\binom{b+1}{2}$. If $k\ge4$, then $d\ge4$ and $2b+d=9$
implies $b\le2$, contradicting $\binom{b+1}{2}\le3$.
\end{proof}

\begin{lemma}\label{qf:compression}
The sets in \eqref{qf:partition-def} satisfy
\begin{equation}\label{qf:compressed-count}
 k+h\le\rho.
\end{equation}
Consequently
\begin{equation}\label{qf:compressed-arithmetic}
 4\delta-\rho\le 2b(b+1)-k+3h.
\end{equation}
\end{lemma}
\begin{proof}
Suppose $x\in A_i$, $y\in A_j$, and $z=x+y\in A_{i+j-1}$.
The level intervals give
\[
 (i+j)m\le z+2\rho<(i+j)m+\rho.
\]
Since $0\le\rho<m$, the residue of $z+2\rho$ modulo $m$ belongs
to $[0,\rho)$. This also shows that no such sum exists if $\rho=0$.

Each element of $K$ is a sum from $B+B$ at level three, so the
observation applies with $(i,j)=(2,2)$. Each element of $H$ is a sum
from $B+C$ at level four, so it applies with $(i,j)=(2,3)$.
The shifted residues of every element of $K\cup H$ thus lie in the
same set of $\rho$ residues. Distinct elements of $A$ have distinct
residues modulo $m$ by Lemma~\ref{lem:apery}; adding $2\rho$ preserves
this distinctness. Since $K$ and $H$ are disjoint, $k+h\le\rho$.

Using \eqref{qf:partition-card} and \eqref{qf:basic-counts}, we obtain
\[
 4\delta-\rho
 \le 3k+4\ell+3h
 =4(k+\ell)-k+3h
 \le 2b(b+1)-k+3h.
\]
\end{proof}

\begin{lemma}\label{qf:cubic-empty}
The remaining set $H$ satisfies
\begin{equation}\label{qf:cubic-bound}
 h\le b(d-k)+f(k),
\end{equation}
where $f$ is the function in Lemma~\ref{lem:shadow}.
\end{lemma}
\begin{proof}
Let $C_{\mathrm{prim}}=C\setminus K$, a set of $d-k$ primitive
elements, and put
\[
 H_3=H\setminus(B+C_{\mathrm{prim}}).
\]
Since $H\subseteq B+C$ and $K\subseteq B+B$, every element of $H_3$
is a sum of three elements of $B$. Consider any representation
$x+y+z\in H_3$ with $x,y,z\in B$. By Ap\'ery factor closure,
$x+y\in A$. Its level is at least $2+2-1=3$.
On the other hand, $x+y+z$ has level four and $z$ has level two, so
\eqref{eq:level-add} gives
$4\ge\lambda(x+y)+2-1$, forcing $\lambda(x+y)\le3$.
Thus $x+y\in A_3\cap D=K$. The same argument applies to the other
two pair sums, for every representation.

Lemma~\ref{lem:shadow} applies because $b\le4$ and $k\le3$, and
gives $|H_3|\le f(k)$. Since
$H\subseteq (B+C_{\mathrm{prim}})\cup H_3$, we conclude
$h\le b(d-k)+f(k)$.
\end{proof}

\begin{lemma}\label{qf:empty}
If $A_1=\varnothing$, then $4\delta\le9e+\rho$.
\end{lemma}
\begin{proof}
For fixed $b,d,k$, set
\[
 H_{\max}=\min\{bd,\ b(d-k)+f(k)\},\qquad
 E_{\min}=\max\{4,\ 1+b+d-k\},
\]
and $F_{\max}=2b(b+1)-k+3H_{\max}$.
The preceding lemmas imply
\[
 h\le H_{\max},\qquad e\ge E_{\min},\qquad
 4\delta-\rho\le F_{\max}.
\]
All allowed parameters occur in the following table. It lists every
$b\in\{0,1,2,3,4\}$ with $d=9-2b$ and every integer
$0\le k\le\min\{d,\binom{b+1}{2},3\}$.
\[
\begin{array}{rrr|rr|rr}
 b&d&k&H_{\max}&E_{\min}&F_{\max}&9E_{\min}\\\hline
 0&9&0&0&10&0&90\\
 1&7&0&7&9&25&81\\
 1&7&1&7&8&24&72\\
 2&5&0&10&8&42&72\\
 2&5&1&9&7&38&63\\
 2&5&2&8&6&34&54\\
 2&5&3&8&5&33&45\\
 3&3&0&9&7&51&63\\
 3&3&1&7&6&44&54\\
 3&3&2&5&5&37&45\\
 3&3&3&4&4&33&36\\
 4&1&0&4&6&52&54\\
 4&1&1&1&5&42&45
\end{array}
\]
Each row follows by substitution of $f(0),f(1),f(2),f(3)=0,1,2,4$.
In every row $F_{\max}\le9E_{\min}$. Therefore
$4\delta-\rho\le9e$, as required.
\end{proof}

\begin{proposition}\label{qf:main}
If $n=13$, $q=4$, and $e\ge4$, then $c\le13e$.
\end{proposition}
\begin{proof}
When $A_1\ne\varnothing$, Lemma~\ref{qf:nonempty} gives
$4\delta\le9e\le9e+\rho$. When $A_1=\varnothing$,
Lemma~\ref{qf:empty} gives the same required inequality directly.
The conclusion follows from \eqref{eq:defect}.
\end{proof}

\section{Depth at most three}\label{sd:section}

Throughout this section, $n=13$.  We use the Ap\'ery and level notation
introduced above.  For finite sets of integers, write
$X+Y=\{x+y:x\in X,\ y\in Y\}$, $2X=X+X$, and $3X=(X+X)+X$.
In particular, $2X$ denotes a sumset, rather than a dilation.

\begin{lemma}\label{sd:low-levels}
Every element of $A_1$ is primitive.  Every decomposable element of
$A_2$ is a sum of two elements of $A_1$.
\end{lemma}
\begin{proof}
Every positive semigroup element is at least $m$.  An element $z\in A_1$
satisfies $z+\rho<2m$, whereas a sum of two positive semigroup elements
is at least $2m$.  Thus $z$ is primitive.

Suppose that $z\in A_2$ and $z=x+y$ with $x,y\in S\setminus\{0\}$.
The factor property of the Ap\'ery set, recalled in
Lemma~\ref{lem:apery}, gives $x,y\in A$.  Since $x,y\ge m$ and
$x+y+\rho<3m$, we have
\[
 m\le x+\rho<2m,\qquad m\le y+\rho<2m.
\]
Hence $x,y\in A_1$.
\end{proof}

\begin{proposition}\label{sd:depth-two}
If $q\le2$, then $c\le13e$.
\end{proposition}
\begin{proof}
Lemma~\ref{lem:weights} gives $q\ge1$. If $q=1$, its weighted
identity gives $n=1$, contrary to $n=13$. For $q=2$, the same
identity reads $13=2+\alpha_1$, so $\alpha_1=11$.
The eleven elements of $A_1$, together with $m$, are distinct primitives;
here $m\notin A$.  Consequently $e\ge12$.

Every decomposition of an element of $D$ has both positive factors in
$T=A_1$, by Lemma~\ref{lem:apery}.  Counting unordered pairs with
repetition therefore gives
\[
 \delta\le |2A_1|\le\binom{12}{2}=66.
\]
It follows that
\[
 c=2(e+\delta)-\rho\le2e+132\le13e,
\]
where the last inequality uses $e\ge12$.
\end{proof}

We now assume $q=3$.  Set
\begin{equation}\label{sd:sets}
 \begin{gathered}
 a=\alpha_1,\qquad b=\alpha_2,\qquad
 K=A_2\cap D,\qquad C=A_2\setminus K,\\
 k=|K|,\qquad p=|C|,\qquad
 H=(3A_1)\cap A,\qquad U=(2C)\cap A,\qquad
 h=|H|,\quad u=|U|.
 \end{gathered}
\end{equation}
The weighted identity and Lemma~\ref{sd:low-levels} imply
\begin{equation}\label{sd:profile}
 2a+b=10,\qquad k+p=b,\qquad
 K\subseteq2A_1,\qquad
 k\le\binom{a+1}{2}.
\end{equation}
All elements of $C$ are primitive by its definition.  The sets
$\{m\}$, $A_1$, and $C$ are pairwise disjoint sets of primitives, so
\begin{equation}\label{sd:primitives}
 e\ge1+a+p.
\end{equation}

\begin{lemma}\label{sd:cover}
With the notation in \eqref{sd:sets},
\begin{equation}\label{sd:delta-bound}
 \delta\le\binom{a+1}{2}+ap+h+u,
 \qquad u\le\rho,\qquad 2u\le p(p+1).
\end{equation}
\end{lemma}
\begin{proof}
We first prove the covering inclusion
\begin{equation}\label{sd:cover-inclusion}
 D\subseteq 2A_1\ \cup\ (A_1+C)\ \cup\ H\ \cup\ U.
\end{equation}
Write $z\in D$ as $z=x+y$ with both factors positive.  Both belong to
$T=A_1\cup A_2$.  If both lie in $A_1$, then $z\in2A_1$.
If one lies in $A_1$ and the other lies in $C$, then $z\in A_1+C$.
If one lies in $A_1$ and the other lies in $K$, split the latter into
two elements of $A_1$ using Lemma~\ref{sd:low-levels}; then $z\in H$.

It remains to consider $x,y\in A_2$.  If, for example, $x$ were
decomposable, then $x\ge2m$ and $y\ge2m-\rho$.  This would give
\[
 z\ge4m-\rho=c+m,
\]
contrary to the upper bound for an Ap\'ery element.  Thus both factors
belong to $C$, and $z\in U$.  This proves \eqref{sd:cover-inclusion}.
Taking cardinalities, and allowing overlaps in this cover, gives the
first inequality in \eqref{sd:delta-bound}.

For $z\in U$, its two factors in $C\subseteq A_2$ give
$z\ge4m-2\rho$.  The Ap\'ery bound gives $z<4m-\rho$.
Thus translation by $2\rho-4m$ injects $U$ into the integer interval
$[0,\rho)$, which proves $u\le\rho$.  Finally,
$U\subseteq2C$ and unordered-pair counting give
$u\le\binom{p+1}{2}$.
\end{proof}

\begin{lemma}\label{sd:cubics}
The integer $k$ lies in $\{0,1,2,3,4\}$, and
\begin{equation}\label{sd:cubic-bound}
 h\le f(k),\qquad
 f(0)=0,\quad f(1)=1,\quad f(2)=2,\quad f(3)=4,\quad f(4)=7.
\end{equation}
\end{lemma}
\begin{proof}
The bounds in \eqref{sd:profile} imply $0\le a\le5$ and
$k\le\min\{10-2a,\binom{a+1}{2}\}$.  If $k\ge5$, then
$b\ge5$, hence $a\le2$ and $k\le3$, a contradiction.  Moreover,
$k>0$ implies $a>0$ because $K\subseteq2A_1$, and $b>0$ then
implies $a\le4$.  If $k=4$, these same inequalities force
$a=3$, $b=4$, and $p=0$.

We verify the divisor condition needed for the small-shadow lemma.
Suppose $x,y,t\in A_1$ and $x+y+t\in H$.  The factor property of
the Ap\'ery set gives $x+y\in A$, and the positivity of $t$ gives
$x+y<c$.  Also $x+y\ge2m$.  Since $c=3m-\rho$, we obtain
\[
 2m\le x+y+\rho<3m.
\]
Thus $x+y\in A_2$, and it is decomposable, so $x+y\in K$.
By symmetry this holds for every pair of factors in every
representation of every element of $H$ as a sum of three elements
of $A_1$.

If $k=0$, any such representation would yield an element of the empty
set $K$.  Hence $H$ is empty.  If $1\le k\le3$, we have
$1\le|A_1|\le4$, $K\subseteq2A_1$, $H\subseteq3A_1$, and the
divisor condition just proved.  The numerical form of
Lemma~\ref{lem:shadow} therefore gives the bounds $1,2,4$.

Finally suppose $k=4$, so $a=3$.  Choose $z_0\in K$ and put
\[
 K'=K\setminus\{z_0\},\qquad
 H'=H\setminus(z_0+A_1).
\]
Then $|K'|=3$.  In any representation $x+y+t\in H'$ with
$x,y,t\in A_1$, each pair sum belongs to $K$ and cannot equal $z_0$:
for instance, $x+y=z_0$ would put the represented value in
$z_0+A_1$.  Thus the divisor condition holds with $K'$ and $H'$.
Lemma~\ref{lem:shadow} gives $|H'|\le4$, while
$|z_0+A_1|=3$.  Consequently $h\le4+3=7$.
\end{proof}

\begin{proposition}\label{sd:depth-three}
If $q=3$, then $c\le13e$.
\end{proposition}
\begin{proof}
We finish with explicit arithmetic for the profiles in
\eqref{sd:profile}.  For an admissible pair $(a,k)$ set
$b=10-2a$, $p=b-k$, and define
\[
 R(a,k)=10(1+a+p)
 -\left(3\binom{a+1}{2}+3ap+3f(k)+p(p+1)\right).
\]
There are precisely sixteen possible quadruples $(a,b,k,p)$ under
the necessary constraints \eqref{sd:profile}.  To enumerate them,
take $a=0,1,\ldots,5$, and then take every integer
$0\le k\le\min\{10-2a,\binom{a+1}{2}\}$; the other two entries
are determined.  Substitution gives the following complete table.
\[
\begin{array}{c|r|r|r|r|r}
 a&b&k&p&f(k)&R(a,k)\\\hline
 0&10&0&10&0&0\\
 1&8&0&8&0&1\\
 1&8&1&7&1&7\\
 2&6&0&6&0&3\\
 2&6&1&5&1&8\\
 2&6&2&4&2&11\\
 2&6&3&3&4&9\\
 3&4&0&4&0&6\\
 3&4&1&3&1&10\\
 3&4&2&2&2&12\\
 3&4&3&1&4&9\\
 3&4&4&0&7&1\\
 4&2&0&2&0&10\\
 4&2&1&1&1&13\\
 4&2&2&0&2&14\\
 5&0&0&0&0&15
\end{array}
\]
Every entry in the last column is nonnegative.  Since
$3u=2u+u\le p(p+1)+\rho$, the preceding lemmas yield
\begin{align*}
 3\delta
 &\le3\binom{a+1}{2}+3ap+3f(k)+p(p+1)+\rho\\
 &\le10(1+a+p)+\rho\le10e+\rho.
\end{align*}
Using $m=e+\delta$ and $c=3m-\rho$ now gives $c\le13e$.
\end{proof}

\begin{theorem}\label{sd:theorem}
If $n=13$ and $q\le3$, then $c\le13e$.
\end{theorem}
\begin{proof}
Combine Propositions~\ref{sd:depth-two} and~\ref{sd:depth-three}.
\end{proof}

\section{Depth at least five}\label{hd:section}

We continue to assume $n=13$, and in this section assume $e\ge4$.
Let $M$ be the eligible unordered-pair count from Lemma~\ref{lem:pairs},
where repetitions are permitted.  Define also the count with a lower
level threshold,
\[
 U=\#\{\{x,y\}:x,y\in T,\ \lambda(x)+\lambda(y)\le q\}.
\]
Thus the threshold defining $M$ is $q+1$, whereas the threshold
defining $U$ is $q$.  We have $\delta\le M$.  Our aim, using
\eqref{eq:defect}, is to prove
\begin{equation}\label{hd:target}
 q\delta\le(13-q)e+\rho.
\end{equation}

\begin{lemma}\label{hd:compression}
For every depth $q$,
\[
 \delta\le U+\rho.
\]
\end{lemma}
\begin{proof}
Consider the subset
\[
 \Gamma=\{z\in D:(q+1)m\le z+2\rho\}.
\]
An element $z\in\Gamma$ satisfies $z<c+m=(q+1)m-\rho$, so
\[
 0\le z+2\rho-(q+1)m<\rho.
\]
The displayed translation is injective, whence $|\Gamma|\le\rho$.

Given $z\in D$, choose positive factors $x,y\in T$ with $z=x+y$.
If $\lambda(x)+\lambda(y)\le q$, then $z$ is a value represented
by one of the $U$ lower pairs.  Otherwise the level lower bounds give
\[
 z+2\rho\ge\bigl(\lambda(x)+\lambda(y)\bigr)m\ge(q+1)m,
\]
so $z\in\Gamma$.  The union of the lower-pair values and $\Gamma$
therefore covers $D$.  Its cardinality is at most $U+\rho$.
\end{proof}

\begin{proposition}\label{hd:depth-five}
If $q=5$, then $5\delta\le8e+\rho$.
\end{proposition}
\begin{proof}
Write $(a,b,d,t)=(\alpha_1,\alpha_2,\alpha_3,\alpha_4)$.
The profile identity and the two pair counts are
\begin{align}
 4a+3b+2d+t&=8,\label{hd:five-profile}\\
 M&=\binom{a+b+d+1}{2}+(a+b)t,\label{hd:five-M}\\
 U&=\binom{a+b+1}{2}+(a+b)d+at.\label{hd:five-U}
\end{align}
Indeed, all pairs among levels $1,2,3$ are eligible for $M$, and a
level-$4$ element can be paired precisely with levels $1$ and $2$.
For $U$, all pairs among levels $1,2$ are permitted; level $3$ can
be paired with levels $1,2$, and level $4$ only with level $1$.
These classes of unordered pairs are disjoint, giving the formulas.

All nonnegative integer solutions of \eqref{hd:five-profile} are
listed below.  Exhaustiveness follows by taking $0\le a\le2$,
$0\le b\le2$, and $0\le d\le4$, and retaining precisely those
choices for which $t=8-4a-3b-2d$ is nonnegative.
\[
\begin{array}{c|r|r}
 (a,b,d,t)&M&U\\\hline
 (0,0,0,8)&0&0\\
 (0,0,1,6)&1&0\\
 (0,0,2,4)&3&0\\
 (0,0,3,2)&6&0\\
 (0,0,4,0)&10&0\\
 (0,1,0,5)&6&1\\
 (0,1,1,3)&6&2\\
 (0,1,2,1)&7&3\\
 (0,2,0,2)&7&3\\
 (0,2,1,0)&6&5\\
 (1,0,0,4)&5&5\\
 (1,0,1,2)&5&4\\
 (1,0,2,0)&6&3\\
 (1,1,0,1)&5&4\\
 (2,0,0,0)&3&3
\end{array}
\]
If $M\le6$, then $5\delta\le30\le8e$.
For either of the profiles $(0,2,0,2)$ and $(0,1,2,1)$, combine
$\delta\le M=7$ with Lemma~\ref{hd:compression} and $U=3$:
\[
 5\delta=4\delta+\delta\le4M+U+\rho=31+\rho\le8e+\rho.
\]

The remaining profile is $(0,0,4,0)$, with $M=10$ and $U=0$.
Here $T=A_3$ has four elements, all primitive.  To see this, a
decomposition of an element $z\in T$ would have two positive factors
in $T$, hence both in $A_3$.  It would imply
\[
 z\ge6m-2\rho>5m-\rho=c,
\]
since $\rho<m$, contradicting $z\in T$.  Thus these four primitives,
together with $m$, give $e\ge5$.  The same two bounds for $\delta$
now give
\[
 5\delta\le4M+U+\rho=40+\rho\le8e+\rho.
\]
This treats every profile.
\end{proof}

\begin{proposition}\label{hd:depth-six}
If $q=6$, then $M\le4$ and $6\delta\le7e$.
\end{proposition}
\begin{proof}
Put $(a,b,d,t,u)=(\alpha_1,\ldots,\alpha_5)$.  The profile identity is
\[
 5a+4b+3d+2t+u=7.
\]
Eligible pairs are all pairs among levels $1,2,3$, pairs between
level $4$ and levels $1,2,3$, and pairs between level $5$ and levels
$1,2$.  Consequently
\[
 M=\binom{a+b+d+1}{2}+(a+b+d)t+(a+b)u.
\]
The full set of profiles and resulting counts is
\[
\begin{array}{c|r}
 (a,b,d,t,u)&M\\\hline
 (0,0,0,0,7)&0\\
 (0,0,0,1,5)&0\\
 (0,0,0,2,3)&0\\
 (0,0,0,3,1)&0\\
 (0,0,1,0,4)&1\\
 (0,0,1,1,2)&2\\
 (0,0,1,2,0)&3\\
 (0,0,2,0,1)&3\\
 (0,1,0,0,3)&4\\
 (0,1,0,1,1)&3\\
 (0,1,1,0,0)&3\\
 (1,0,0,0,2)&3\\
 (1,0,0,1,0)&2
\end{array}
\]
For exhaustiveness, the profile equation forces
$a\le1$, $b\le1$, $d\le2$, and $t\le3$; choosing these four
entries and requiring $u=7-5a-4b-3d-2t\ge0$ gives exactly the
thirteen rows.  Thus $\delta\le M\le4$, and
$6\delta\le24\le7e$ because $e\ge4$.
\end{proof}

\begin{proposition}\label{hd:depth-seven-up}
For the remaining depths the following bounds hold:
\[
\begin{array}{c|c}
 q&M\\\hline
 7&M\le3\\
 8,9&M\le1\\
 10\le q\le13&M=0.
\end{array}
\]
In each case $q\delta\le(13-q)e$.
\end{proposition}
\begin{proof}
For $x\in T$, set $w(x)=q-\lambda(x)$.  These are positive integers
with total
\[
 B=\sum_{x\in T}w(x)=13-q.
\]
In particular $q\le13$.  A pair is eligible precisely when the sum
of its two weights is at least $q-1$.  If $x,y\in T$ are distinct
and $t=|T|$, positivity of all the remaining weights gives
\begin{equation}\label{hd:distinct-weights}
 w(x)+w(y)\le B-(t-2).
\end{equation}

Suppose $q=7$, so $B=6$ and the threshold is $6$.  If $t\le2$,
there are at most $\binom{t+1}{2}\le3$ unordered pairs in total.
If $t\ge3$, \eqref{hd:distinct-weights} makes the weight sum of
every distinct pair at most $5$, so only loops can be eligible.
Two eligible loops at distinct vertices would each have weight at
least $3$; together with a third positive weight their total would
exceed $B=6$.  Hence there is at most one eligible loop in this
case.  In either case $M\le3$, and
$7\delta\le21\le6e$.

For $q=8$ or $q=9$, the total weight is respectively $5$ or $4$,
strictly less than the respective thresholds $7$ or $8$.  Thus no
distinct pair is eligible.  Two distinct eligible loops would
satisfy
\[
 2w(x)\ge q-1,\qquad 2w(y)\ge q-1,
\]
and hence $w(x)+w(y)\ge q-1>B$, which is impossible.
Therefore $M\le1$.  This yields $8\delta\le8\le5e$ for $q=8$,
and $9\delta\le9\le4e$ for $q=9$.

Finally let $q\ge10$.  Every single weight is at most $B$, so the
weight sum of any pair, including a loop, is at most
$2B=2(13-q)<q-1$.  There are no eligible pairs and $\delta=0$.
Together with $q\le13$, this gives the stated inequality.
\end{proof}

\begin{theorem}\label{hd:theorem}
If $n=13$, $e\ge4$, and $q\ge5$, then $c\le13e$.
\end{theorem}
\begin{proof}
Proposition~\ref{hd:depth-five} proves \eqref{hd:target} for $q=5$.
Propositions~\ref{hd:depth-six} and~\ref{hd:depth-seven-up} give the
stronger inequality $q\delta\le(13-q)e$ for every $q\ge6$.
Now use
\[
 c=q(e+\delta)-\rho\le13e.
\]
\end{proof}



\section{Small embedding dimension: an exact finite evaluation}
\label{sec:small-embedding}

We prove the case $e\leq 3$ by reducing it to a bounded collection of
generator tuples.  Throughout this section $n=13$.  The primitive set is
the full set $P$, including primitives at or above the conductor.  We use
$M$ for the number of eligible unordered pairs from $T$, with loops counted
once, as in Lemma~\ref{lem:pairs}.

\begin{lemma}\label{se:generation}
Every element of $S$ is a sum of primitive elements, where the empty sum is
$0$.  In particular, if $P=\{m,a,b\}$, then
\begin{equation}\label{se:linear-combination}
 x\in S\quad\Longleftrightarrow\quad
 x=im+ja+kb\text{ for some }i,j,k\in\mathbb N.
\end{equation}
The assertion allows $a=b$.
\end{lemma}
\begin{proof}
Use strong induction on $x\in S$.  The assertion holds for $x=0$ and for
primitive $x$.  Otherwise $x=u+v$ for positive $u,v\in S$.  Both summands
are smaller than $x$, so their primitive decompositions concatenate to
give one for $x$.  This proves the forward implication in
\eqref{se:linear-combination}; the reverse implication follows from
closure under addition.
\end{proof}

\begin{lemma}\label{se:bounds}
If $e\leq3$, then
\[
 4\leq q\leq13,\qquad m\leq18,\qquad c\leq72,
 \qquad c+m\leq90.
\]
\end{lemma}
\begin{proof}
Suppose first that $q\leq3$, so $c\leq3m$.  If a positive element
$x<c$ is decomposable, write $x=u+v$ with $u,v>0$.  Neither summand can
be decomposable: a decomposition of one would express $x$ as a sum of
three positive elements of $S$, each at least $m$.  Thus
\[
 S\cap[0,c)\subseteq\{0\}\cup P\cup(P+P),
 \qquad
 n\leq1+e+\binom{e+1}{2}\leq10,
\]
contrary to $n=13$.  Hence $q\geq4$.  Lemma~\ref{lem:weights} gives
$q\leq n=13$.

The identities and estimates from Lemmas~\ref{lem:apery}--\ref{lem:pairs}
give
\begin{equation}\label{se:basic-bounds}
 m=e+\delta\leq3+M,\qquad c\leq qm,
 \qquad (q-1)M\leq(14-q)(13-q).
\end{equation}
For depths $4,5,6$, we use the exact profile identities.  In each line
below, $u,v,w,t,s$ denote the successive positive-level cardinalities
$\alpha_1,\alpha_2,\ldots$, as far as needed:
\[
\begin{aligned}
 q=4:\quad&3u+2v+w=9,\\
 &M=\binom{u+v+1}{2}+(u+v)w;\\[.4ex]
 q=5:\quad&4u+3v+2w+t=8,\\
 &M=\binom{u+v+w+1}{2}+(u+v)t;\\[.4ex]
 q=6:\quad&5u+4v+3w+2t+s=7,\\
 &M=\binom{u+v+w+1}{2}+(u+v+w)t+(u+v)s.
\end{aligned}
\]
These identities give $M\leq15,10,7$, respectively.  Here is an explicit
way to check this small arithmetic step.  Eliminate the last variable
using the weight equation, and range the remaining variables over the
nonnegative integers for which that last variable is nonnegative.
At $q=4$, take $0\leq u\leq3$ and
$0\leq v\leq\lfloor(9-3u)/2\rfloor$.
At $q=5$, take $0\leq u\leq2$,
$0\leq v\leq\lfloor(8-4u)/3\rfloor$, and
$0\leq w\leq\lfloor(8-4u-3v)/2\rfloor$.
At $q=6$, it suffices to take $0\leq u,v\leq1$, $0\leq w\leq2$,
and $0\leq t\leq3$, retaining only those choices with
$5u+4v+3w+2t\leq7$.
Substitution gives the following maxima for fixed $u$:
\[
\begin{array}{c|rrrr}
 &u=0&u=1&u=2&u=3\\ \hline
 q=4&15&12&9&6\\
 q=5&10&6&3&\text{--}\\
 q=6&4&3&\text{--}&\text{--}
\end{array}
\]
In particular, the stated bound $M\leq7$ at depth six suffices.
For $7\leq q\leq13$, the final inequality in
\eqref{se:basic-bounds} gives the integer bound
\[
 M\leq\left\lfloor\frac{(14-q)(13-q)}{q-1}\right\rfloor.
\]
Combining these bounds with $m\leq3+M$ and $c\leq qm$ gives
\[
\begin{array}{c|rrrrrrrrrr}
 q&4&5&6&7&8&9&10&11&12&13\\ \hline
 M\text{ at most}&15&10&7&7&4&2&1&0&0&0\\
 m\text{ at most}&18&13&10&10&7&5&4&3&3&3\\
 c\text{ at most}&72&65&60&70&56&45&40&33&36&39\\
 c+m\text{ at most}&90&78&70&80&63&50&44&36&39&42
\end{array}
\]
The required uniform bounds follow.
\end{proof}

\begin{lemma}\label{se:parameters}
If $e\leq3$, there are integers $a,b$ such that
\begin{equation}\label{se:parameter-range}
 2\leq m\leq18,\qquad m<a\leq b<90,\qquad
 P=\{m,a,b\},\qquad c\leq72,
\end{equation}
and
\begin{equation}\label{se:dimension-convention}
 e=r(a,b),\qquad
 r(a,b)=\begin{cases}2,&a=b,\\3,&a<b.\end{cases}
\end{equation}
\end{lemma}
\begin{proof}
If $m=1$, then $S=\mathbb N$, $c=0$, and $n=0$.  Thus $m\geq2$.
Also $m\in P$, so $e\geq1$.  If $e=1$, Lemma~\ref{se:generation}
would make every element of $S$ a multiple of $m$.  Since $c,c+1\in S$,
this would imply $m\mid1$, again a contradiction.  Hence $e=2$ or $3$.

When $e=2$, write $P=\{m,a\}$ and put $b=a$.  When $e=3$, order the
other two primitives as $a<b$.  Their positivity and the definition of
multiplicity give $m<a\leq b$.  Every primitive different from $m$ belongs
to $A$, and every element of $A$ is strictly smaller than $c+m$.
Lemma~\ref{se:bounds} therefore gives $a,b<90$.  The remaining bounds
follow from the same lemma.  The duplicated entry in the case $a=b$
does not add an element to the set $P$, which proves
\eqref{se:dimension-convention}.
\end{proof}

We next specify the computation used on the bounded tuples.
All loop bounds in the pseudocode below are inclusive.  Division and
remainders are integer operations on nonnegative integers.

\begin{Verbatim}
coinMem(m, a, b, x):
    for k = 0, ..., floor(x / b):
        for j = 0, ..., floor((x - k*b) / a):
            if (x - k*b - j*a) mod m = 0:
                return true
    return false

twoCoinMem(m, a, b):
    for j = 0, ..., floor(b / a):
        if (b - j*a) mod m = 0:
            return true
    return false
\end{Verbatim}

\begin{lemma}\label{se:membership}
For positive $m,a,b$ and $x\in\mathbb N$,
\[
 \operatorname{coinMem}(m,a,b,x)=\mathrm{true}
 \quad\Longleftrightarrow\quad
 x=im+ja+kb\quad\text{for some }i,j,k\in\mathbb N.
\]
Consequently, for the tuple in Lemma~\ref{se:parameters}, this test is
equivalent to $x\in S$, for every $x\in\mathbb N$.
Likewise, a true value of $\operatorname{twoCoinMem}(m,a,b)$ implies
$b=im+ja$ for some $i,j\in\mathbb N$.
\end{lemma}
\begin{proof}
The loop bounds ensure $kb\leq x$ and $ja\leq x-kb$.
Thus every remainder tested by the first algorithm is nonnegative.
If it is divisible by $m$, it equals $im$ for some $i\in\mathbb N$,
giving the required representation.  Conversely, a representation
$x=im+ja+kb$ puts $k$ and $j$ inside the two stated loop ranges, and
the tested remainder is $im$.  Apply Lemma~\ref{se:generation} for the
semigroup interpretation.  The same argument with one loop proves the
assertion about $\operatorname{twoCoinMem}$.
\end{proof}

For a tuple with $2\leq m<a\leq b<90$, the complete checker is as follows.
Here $v[x]$ denotes a Boolean table entry.  The finite implementation
stores this table as the natural number
$\sum_{x=0}^{89}2^x\,\mathbf{1}_{v[x]=\mathrm{true}}$;
the elementary identity
\[
\left\lfloor 2^{-x}\sum_{y=0}^{89}2^y\mathbf1_{v[y]=\mathrm{true}}
\right\rfloor\bmod2=\mathbf1_{v[x]=\mathrm{true}}
\qquad(0\le x<90)
\]
identifies its bit at $x$ with $v[x]$: lower powers sum to less than
$2^x$, while higher powers contribute an even integer after division.

\begin{Verbatim}
coinCheck(m, a, b):
    r = 2 if a = b, otherwise 3
    if m <= r: return true
    if gcd(gcd(m,a),b) != 1: return true
    if a mod m = 0: return true
    if a < b and twoCoinMem(m,a,b): return true
    if some x in 72, ..., 89 has not coinMem(m,a,b,x):
        return true

    for x = 0, ..., 89:
        v[x] = coinMem(m,a,b,x)
    cut = 0
    for x = 0, ..., 89:
        if not v[x]: cut = x + 1
    left = number of x in 0, ..., cut-1 with v[x] = true
    return (cut > 72) or (left != 13) or (cut <= 13*r)
\end{Verbatim}

The range used for $\texttt{left}$ is empty when $\texttt{cut}=0$.
The early returns permit tuples that need no further computation.
The following argument accounts for each such return on a tuple obtained
from an actual semigroup.

\begin{proposition}\label{se:checker-sound}
Let $S$ satisfy $n=13$ and $e\leq3$, and choose $a,b$ as in
Lemma~\ref{se:parameters}.  If
$\operatorname{coinCheck}(m,a,b)=\mathrm{true}$, then $c\leq13e$.
\end{proposition}
\begin{proof}
We have $r(a,b)=e$.  If $m\leq e$, then
$c\leq qm\leq13m\leq13e$, so the first early return is justified.
Suppose henceforth that $m>e$.

The other four early returns cannot occur for this tuple.  Indeed, any
common divisor of $m,a,b$ divides every element of $S$ by
Lemma~\ref{se:generation}, and therefore divides both $c$ and $c+1$.
Hence $\gcd(m,a,b)=1$.  If $a\bmod m=0$, then
$a>m$ expresses $a$ as a sum of at least two copies of $m$, contrary to
primitivity.  If $a<b$ and $\operatorname{twoCoinMem}(m,a,b)$ is true,
Lemma~\ref{se:membership} expresses $b$ as a nonnegative combination of
$m,a$.  Since both are strictly smaller than $b$, that representation
uses at least two positive summands, again contradicting primitivity.
Finally, $c\leq72$ implies that every integer from $72$ through $89$
belongs to $S$, so all eighteen membership tests in the tail guard are
true.  The condition $a<b$ in the redundancy guard is essential to
retain the convention $a=b$ when $e=2$.

It remains to interpret the Boolean table.  By Lemma~\ref{se:membership},
$v[x]$ is true exactly when $x\in S$ for $0\leq x<90$.
Let $d$ be the computed value of $\texttt{cut}$.
All gaps in this interval lie below $c$, so $d\leq c$.
By the definition of the last-gap scan, every $x$ with $d\leq x<90$
belongs to $S$.  Every $x\geq90$ also belongs to $S$, since $c\leq72$.
Minimality of the conductor now gives $c\leq d$, proving $d=c$.
It follows that $\texttt{left}=|S\cap[0,c)|=n=13$, with zero included
in the count.  The first two disjuncts of the final return are therefore
false, and its true value implies $c\leq13r(a,b)=13e$.
\end{proof}

\begin{lemma}[Exact finite evaluation]\label{se:finite-check}
For all integers satisfying
\[
 2\leq m\leq18,\qquad m<a\leq b<90,
\]
we have $\operatorname{coinCheck}(m,a,b)=\mathrm{true}$.
\end{lemma}
\begin{proof}
For each fixed $m,a$ with $2\leq m\leq18$ and $m<a<90$, form the
closed Boolean expression
\[
 R(m,a)=\bigwedge_{b=0}^{89}
 \begin{cases}
  \operatorname{coinCheck}(m,a,b),&a\leq b,\\
  \mathrm{true},&b<a.
 \end{cases}
\]
For an equivalent evaluation of every membership test in the checker,
set $v_0=1$ and, for $1\le x<90$, use the recurrence
\begin{equation}\label{eq:coin-recurrence}
v_x=\bigvee_{\substack{d\in\{m,a,b\}\\d\le x}}v_{x-d}.
\end{equation}
Induction on $x$ gives $v_x=1$ exactly when
$x\in\langle m,a,b\rangle$, so \eqref{eq:coin-recurrence} and
\texttt{coinMem} compute the same table. Let
\[
f_m=\sum_{a=m+1}^{89}\sum_{b=a}^{89}
\bigl[\operatorname{coinCheck}(m,a,b)=\mathrm{false}\bigr].
\]
The following complete outer evaluator computes the displayed tally.
The row value $R(m,a)$ is the conjunction accumulated by the inner
loop. The calls to \texttt{coinCheck} have the exact meaning given
above; its membership table may equivalently use
\eqref{eq:coin-recurrence}.
\begin{Verbatim}
allRows = 0; allTriples = 0; allFailures = 0
for m = 2, ..., 18:
    rows = 0; triples = 0; failures = 0
    for a = m+1, ..., 89:
        rows = rows + 1
        rowValue = true; rowFailures = 0
        for b = a, ..., 89:
            triples = triples + 1
            z = coinCheck(m,a,b)
            rowValue = rowValue and z
            if not z:
                failures = failures + 1
                rowFailures = rowFailures + 1
        assert rowValue = (rowFailures = 0)
    output (m,rows,triples,failures)
    allRows = allRows + rows
    allTriples = allTriples + triples
    allFailures = allFailures + failures
output (allRows,allTriples,allFailures)
\end{Verbatim}
All loops terminate at stated bounds. The complete
finite evaluation yields the following tally; the middle
columns also expose the number of rows and nontrivial triples at each
multiplicity.
\begin{center}
\begin{tabular}{rrrr@{\qquad}rrrr}
\toprule
$m$ & rows & triples & $f_m$ & $m$ & rows & triples & $f_m$\\
\midrule
2&87&3828&0 &11&78&3081&0\\
3&86&3741&0 &12&77&3003&0\\
4&85&3655&0 &13&76&2926&0\\
5&84&3570&0 &14&75&2850&0\\
6&83&3486&0 &15&74&2775&0\\
7&82&3403&0 &16&73&2701&0\\
8&81&3321&0 &17&72&2628&0\\
9&80&3240&0 &18&71&2556&0\\
10&79&3160&0 & & & &\\
\bottomrule
\end{tabular}
\end{center}
The finite certificate consists of the closed equalities
$R(m,a)=\mathrm{true}$ for every such pair $(m,a)$. Each equality is
obtained by exact evaluation of the displayed Boolean expression with
the algorithms specified above. These are finite computational steps
of the proof, not sampled instances.

The rows are assembled by finite cases on $a=m+1,\ldots,89$ and
$m=2,\ldots,18$.
For any tuple in the statement, $a\leq b<90$ also gives $a<90$.
Its value $R(m,a)$ is therefore among the certified rows, and the
conjunct indexed by this $b$ is exactly
$\operatorname{coinCheck}(m,a,b)$.  This proves the asserted coverage.

More explicitly, the row count is
\[
 \sum_{m=2}^{18}(89-m)=1343,
\]
and the number of ordered triples to which the nontrivial branch
$a\leq b$ applies is
\[
 \sum_{m=2}^{18}\sum_{a=m+1}^{89}(90-a)
 =\sum_{m=2}^{18}\frac{(89-m)(90-m)}{2}
 =53924.
\]
Thus the certificate includes both boundary multiplicities $2,18$,
the largest allowed generator $89$, and every diagonal case $a=b$.
The lemma is computer-assisted: its row equalities are supplied by the
checked finite computations, and their application to arbitrary
parameters is the finite-case argument just given.
\end{proof}

\begin{proposition}\label{se:main}
If $n=13$ and $e\leq3$, then $c\leq13e$.
\end{proposition}
\begin{proof}
Lemma~\ref{se:parameters} places the primitive generators of $S$ in the
range of Lemma~\ref{se:finite-check}.  That lemma gives a true checker
value, and Proposition~\ref{se:checker-sound} gives the desired bound.
\end{proof}

\section{Completion and a sharp example}\label{sec:completion}

\begin{proof}[Proof of Theorem~\ref{thm:main}]
If $e\le3$, apply Proposition~\ref{se:main}. Suppose $e\ge4$.
Lemma~\ref{lem:weights} gives $1\le q\le13$. For $q\le3$, apply
Theorem~\ref{sd:theorem}; for $q=4$, apply
Proposition~\ref{qf:main}; and for $q\ge5$, apply
Theorem~\ref{hd:theorem}. These cases exhaust the possibilities,
and each gives $c\le13e$.
\end{proof}

The inequality can be an equality. Consider
\[
 S=\langle2,27\rangle
   =\{x\in\mathbb N:x\text{ is even or }x\ge27\}.
\]
All integers at least $26$ belong to $S$, whereas $25$ does not, so
$c=26$. Its left part is $\{0,2,4,\ldots,24\}$, of cardinality $13$.
The element $2$ is primitive by minimality. A decomposition of $27$ into
two positive elements would have an odd summand, which must be at least
$27$, an impossibility. The primitive elements are therefore exactly
$2$ and $27$: all larger even elements
split off $2$, and every odd element above $27$ splits off $2$ as well.
Consequently $e=2$ and $c=13e$. Notice that $27$ is primitive but is not
in the left part. This example illustrates why the embedding dimension
in the theorem must count the complete primitive set.

\section*{Acknowledgements}
AI assistance was used in mathematical exploration, proof development,
computational work, and the preparation and revision of this manuscript.
The author takes responsibility for its content.


\begin{thebibliography}{99}
\bibitem{wilf} H.~S. Wilf, \emph{A circle-of-lights algorithm for the
``money-changing problem''}, Amer. Math. Monthly \textbf{85} (1978),
562--565. \url{https://doi.org/10.1080/00029890.1978.11994639}.
\bibitem{eliahou-marin} S. Eliahou and D. Mar\'in-Arag\'on,
\emph{On numerical semigroups with at most 12 left elements},
Commun. Algebra \textbf{49} (2021), no.~6, 2402--2422.
\url{https://doi.org/10.1080/00927872.2021.1871621}.
\bibitem{miller} A.~W. Miller, \emph{Hechler and Laver trees},
arXiv:1204.5198 (2012). \url{https://arxiv.org/abs/1204.5198}.
\bibitem{khomskii} Y. Khomskii, \emph{Filter-Laver measurability},
Topology Appl. \textbf{228} (2017), 208--221.
\url{https://doi.org/10.1016/j.topol.2017.06.004}.
Author preprint: \url{https://arxiv.org/abs/1612.04170}.
\bibitem{palumbo} J. Palumbo,
\emph{Unbounded and dominating reals in Hechler extensions},
J. Symbolic Logic \textbf{78} (2013), no.~1, 275--289.
\url{https://doi.org/10.2178/jsl.7801190}.
Author preprint: \url{https://arxiv.org/abs/1201.2932}.
\bibitem{martin} D.~A. Martin, \emph{Borel determinacy},
Ann. of Math. (2) \textbf{102} (1975), 363--371.
\url{https://annals.math.princeton.edu/1975/102-2/p06}.
\bibitem{bell} M. Bell, \emph{On the combinatorial principle
$P(\mathfrak c)$}, Fund. Math. \textbf{114} (1981), 149--157.
\url{https://doi.org/10.4064/fm-114-2-149-157}.
\bibitem{hrusakminami} M. Hru\v{s}\'ak and H. Minami,
\emph{Mathias-Prikry and Laver-Prikry type forcing},
Ann. Pure Appl. Logic \textbf{165} (2014), no.~3, 880--894.
\url{https://doi.org/10.1016/j.apal.2013.11.003}.
Author preprint: \url{https://matmor.unam.mx/~michael/preprints_files/HrusakMinami.pdf}.
\bibitem{eliahou-macaulay} S. Eliahou,
\emph{Wilf's conjecture and Macaulay's theorem},
J. Eur. Math. Soc. \textbf{20} (2018), no.~9, 2105--2129.
\url{https://doi.org/10.4171/JEMS/807}.
\bibitem{eliahou-graph} S. Eliahou,
\emph{A graph-theoretic approach to Wilf's conjecture},
Electron. J. Combin. \textbf{27} (2020), no.~2, P2.15.
\url{https://doi.org/10.37236/9106}.
\end{thebibliography}
\end{document}